\documentclass[12pt]{article}

\usepackage[centertags]{amsmath}
\usepackage{amssymb}
\usepackage{amsthm}
\usepackage[misc]{ifsym}
\usepackage{mathrsfs}
\usepackage{graphicx}
\usepackage{epstopdf}
\usepackage{pgf,tikz}
\usetikzlibrary{patterns}
\usetikzlibrary{arrows}
\usepackage{lineno}
\usepackage{authblk}
\usepackage{multicol}

\theoremstyle{plain}
\newtheorem{thm}{Theorem}[section]

\newtheorem{lem}[thm]{Lemma}
\newtheorem{prop}[thm]{Proposition}

\newtheorem{conj}[thm]{Conjecture}

\theoremstyle{definition}
\newtheorem{defn}[thm]{Definition}
\newtheorem{exam}[thm]{Example}
\newtheorem{rem}[thm]{Remark}

\theoremstyle{remark}
\numberwithin{equation}{section}

\usepackage{hyperref}

\begin{document}

\title{The integer $\{2\}$-domination number of grids}

\author[*]{Jia-Ying Lee} 
\author[*]{Chia-An Liu}
\affil[*]{Department of Mathematics, Soochow University, Taipei 111002, Taiwan}
\date{January 31, 2025}

\maketitle

\begin{abstract}
For positive integers $m$ and $n$, the grid graph $G_{m,n}$ is the Cartesian product of the path graph $P_m$ on $m$ vertices and the path graph $P_n$ on $n$ vertices. An integer $\{2\}$-dominating function of a graph is a mapping from the vertex set to $\{0,1,2\}$ such that the sum of the mapped values of each vertex and its neighbors is at least $2$; the integer $\{2\}$-domination number of a graph is defined to be the minimum sum of mapped values of all vertices among all integer $\{2\}$-dominating functions. In this paper, we compute the integer $\{2\}$-domination numbers of $G_{1,n}$ and $G_{2,n}$, attain an upper bound to the integer $\{2\}$-domination numbers of $G_{3,n}$, and propose an algorithm to count the integer $\{2\}$-domination numbers of $G_{m,n}$ for arbitrary $m$ and $n$. As a future work, we list the integer $\{2\}$-domination numbers of $G_{4,n}$ for small $n$, and conjecture on its formula.
\end{abstract}

{\bf keywords}: grid graph, integer $\{2\}$-dominating function, integer $\{2\}$-domination number.

{\bf MSC2010}:  05C69, 05C85

\section{Introduction}

A {\it simple graph} $G=(V,E)$ is composed of a vertex set $V$ and an edge set $E$, where $V$ is a finite set of elements called {\it vertices}, and $E$ is a set of unordered pairs of distinct vertices called {\it edges}. In this study, a graph is always a simple graph, and for an edge $\{u,v\}$ we write $uv$ in brief.

For a positive integer $n$, let $P_n$ be the {\it path graph} with vertex set $\{v_1,v_2,\ldots,v_n\}$, in which two vertices $v_i,v_j$ are adjacent if and only if $|i-j|=1.$ The {\it Cartesian product} of two graphs $G=(V_1,E_1)$ and $H=(V_2,E_2)$, denoted by $G\square H$, is the graph with vertex set
$$\{(u,v)\mid u\in V_1, v\in V_2\},$$
in which two vertices $(u,v)$ and $(u',v')$ are adjacent if and only if $u=u'$ and $vv'\in E_2$ or $uu'\in E_1$ and $v=v'.$ For positive integers $m$ and $n$, the {\it grid graph} $G_{m,n}=P_m\square P_n$ is the Cartesian product of path graphs $P_m$ and $P_n$. For example,  the grid graph $G_{3,7}=P_3\square P_7$ is shown in Figure~\ref{fig_grid37}.

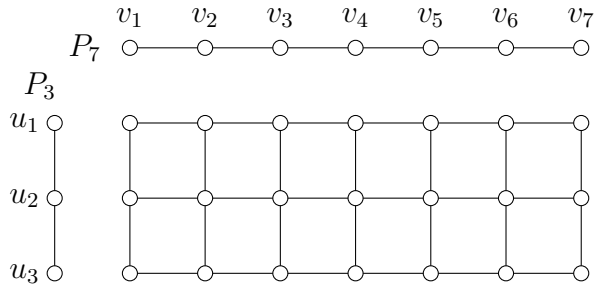
\begin{figure}[h]
\begin{center}
\begin{tikzpicture}

\draw (0,0)--(0,2);
\draw (1,3)--(7,3);
\draw (1,0)--(7,0);
\draw (1,1)--(7,1);
\draw (1,2)--(7,2);
\draw (1,0)--(1,2);
\draw (2,0)--(2,2);
\draw (3,0)--(3,2);
\draw (4,0)--(4,2);
\draw (5,0)--(5,2);
\draw (6,0)--(6,2);
\draw (7,0)--(7,2);

\draw (0.4,3) node {$P_7$};
\draw (-0.2,2.5) node{$P_3$};
\draw (-0.4,0) node{$u_3$};
\draw (-0.4,1) node{$u_2$};
\draw (-0.4,2) node{$u_1$};
\draw (1,3.4) node{$v_1$};
\draw (2,3.4) node{$v_2$};
\draw (3,3.4) node{$v_3$};
\draw (4,3.4) node{$v_4$};
\draw (5,3.4) node{$v_5$};
\draw (6,3.4) node{$v_6$};
\draw (7,3.4) node{$v_7$};

\draw[fill=white] (0,0)coordinate(A00) circle(0.1cm);
\draw[fill=white] (0,1)coordinate(A01) circle(0.1cm);
\draw[fill=white] (0,2)coordinate(A02) circle(0.1cm);
\draw[fill=white] (1,0)coordinate(A10) circle(0.1cm);
\draw[fill=white] (1,1)coordinate(A11) circle(0.1cm);
\draw[fill=white] (1,2)coordinate(A12) circle(0.1cm);
\draw[fill=white] (1,3)coordinate(A13) circle(0.1cm);
\draw[fill=white] (2,0)coordinate(A20) circle(0.1cm);
\draw[fill=white] (2,1)coordinate(A21) circle(0.1cm);
\draw[fill=white] (2,2)coordinate(A22) circle(0.1cm);
\draw[fill=white] (2,3)coordinate(A23) circle(0.1cm);
\draw[fill=white] (3,0)coordinate(A30) circle(0.1cm);
\draw[fill=white] (3,1)coordinate(A31) circle(0.1cm);
\draw[fill=white] (3,2)coordinate(A32) circle(0.1cm);
\draw[fill=white] (3,3)coordinate(A33) circle(0.1cm);
\draw[fill=white] (4,0)coordinate(A40) circle(0.1cm);
\draw[fill=white] (4,1)coordinate(A41) circle(0.1cm);
\draw[fill=white] (4,2)coordinate(A42) circle(0.1cm);
\draw[fill=white] (4,3)coordinate(A43) circle(0.1cm);
\draw[fill=white] (5,0)coordinate(A50) circle(0.1cm);
\draw[fill=white] (5,1)coordinate(A51) circle(0.1cm);
\draw[fill=white] (5,2)coordinate(A52) circle(0.1cm);
\draw[fill=white] (5,3)coordinate(A53) circle(0.1cm);
\draw[fill=white] (6,0)coordinate(A60) circle(0.1cm);
\draw[fill=white] (6,1)coordinate(A61) circle(0.1cm);
\draw[fill=white] (6,2)coordinate(A62) circle(0.1cm);
\draw[fill=white] (6,3)coordinate(A63) circle(0.1cm);
\draw[fill=white] (7,0)coordinate(A70) circle(0.1cm);
\draw[fill=white] (7,1)coordinate(A71) circle(0.1cm);
\draw[fill=white] (7,2)coordinate(A72) circle(0.1cm);
\draw[fill=white] (7,3)coordinate(A73) circle(0.1cm);

\end{tikzpicture}

\caption{The grid graph $G_{3,7}=P_3\square P_7$.}
\label{fig_grid37}

\end{center}
\end{figure}

A set $S$ of vertices of a graph $G=(V,E)$ is called a {\it dominating set}, provided that each vertex in $V\setminus S$ is adjacent to at least one vertex in $S$. The {\it domination number} of $G$ is the minimum size of a dominating set of $G$, and is denoted by $\gamma(G).$ Many authors realize the domination problem by placing one stone in each vertex of $S$ such that all vertices have or are adjacent to at least one stone. In this sense, $\gamma(G)$ denotes the minimum number of stones in a graph if all vertices are dominated. The exact values for $G_{m,n}$ for $m=1,2,3,4$ were found in~\cite{jk83}. In 1992, Chang~\cite{c92} established the domination numbers of $G_{5,n}$ and $G_{6,n}$ in his Ph.D. dissertation. Moreover, he proposed a dynamic program for computing $\gamma(G_{m,n})$ based on an algorithm of Hare given in~\cite{hhh86}. In 2011, Gon\c{c}alves {\it et al.}~\cite{gprt11} concluded the domination number of all grids. More precisely, for small $m,n$ they listed the values of $\gamma(G_{m,n})$, and for $16\leq m\leq n$ they proved
$$\gamma(G_{m,n}) = \left\lfloor\frac{(m+2)(n+2)}{5}\right\rfloor-4,$$
which was also conjectured by Chang in~\cite[Chapter~6]{c92}.

The domination-type problems of the grid graphs are studied by many authors. A {\it total dominating set} $S$ of a graph $G=(V,E)$ is a subset of $V$ such that every vertex in $V$ (instead of $V\setminus S$) is adjacent to at least one vertex in $S$. Also, the {\it total domination number} $\gamma_{t}(G)$ of $G$ is the minimum size of a total dominating set of $G$. In 2002, Gravier~\cite{g02} investigated the value of $\gamma_{t}(G_{m,n})$ and attained some upper and lower bounds of $\gamma_{t}(G_{m,n})$. A {\it 2-dominating set} of a graph $G=(V,E)$ is a subset of $V$ such that every vertex in $V\setminus S$ is adjacent to at least two vertices in $S$, and the {\it 2-domination number} $\gamma_2(G)$ is the minimum size of a 2-dominating set of $G$. In 2017, Shaheen {\it et al.}~\cite{sma17} calculated the value of $\gamma_2(G_{m,n})$ for $m\leq 5$ and arbitrary $n$. However, in 2019 Rao {\it et al.}~\cite{rt19} corrected the exact value of $\gamma_2(G_{m,n})$ computed in~\cite{sma17} and attained $\gamma_2(G_{m,n})$ for all $m,n$ using some methods similar to those in~\cite{gprt11}. A {\it Roman domination set} of a graph $G=(V,E)$ is a pair $(S_1,S_2)$ of disjoint subsets of $V$ such that every vertex in $V\setminus \{S_1\cup S_2\}$ is adjacent to at least one vertex in $S_2$, while the {\it Roman domination number} $\gamma_R(G)$ of $G$ is the minimum number of $|S_1|+2|S_2|$ among all Roman domination sets $(S_1,S_2)$ of $G$. In the same paper~\cite{rt19}, Rao {\it et al.} further proposed a method for obtaining $\gamma_R(G_{m,n}).$

One may notice that in the previous Roman dominating problem, we can put ``two'' stones on a vertex if we see that the sets $S_1$ and $S_2$ collect the vertices having $1$ and $2$ stones, respectively. For any vertex $v$ in a graph, the {\it neighborhood} $N(v)$ of $v$ is the set of vertices adjacent to $v$, and the {\it closed neighborhood} $N[v]$ of $v$ is the set $N(v)\cup\{v\}.$ Let $k$ be a positive integer, and $\mathbb{Z}^{\geq 0}$ be the set of nonnegative integers. An {\it integer $\{k\}$-dominating function} of a graph $G=(V,E)$ is a function $f:V\rightarrow \mathbb{Z}^{\geq 0}$ satisfying
$$\sum_{y\in N[x]}f(y)\geq k$$
for each $x\in V.$ The {\it integer $\{k\}$-domination number} $\gamma_{\{k\}}(G)$ of a graph $G=(V,E)$ is the minimum value of $\sum_{x\in V}f(x)$ among all integer $\{k\}$-dominating functions $f$ of $G$. The problem of computing the integer $\{k\}$-domination number of graphs was started by Domke {\it et al.} in~\cite{dhlf91}. So far, results for the integer $\{k\}$-domination numbers are all on special classes of graphs. Please refer to~\cite{bhk06} on the Cartesian product of graphs, and~\cite{cfl20} on the circulant graphs. When $k=1$, the number $\gamma_{\{1\}}(G)$ is clearly the domination number $\gamma(G)$ of $G.$ We will study the method for computing $\gamma_{\{2\}}(G_{m,n})$ for all grid graphs $G_{m,n}$.

Throughout this paper, a grid graph $G_{m,n}$ has the vertex set
$$V_{m,n}=\{v_{i,j}\mid 1\leq i\leq m,1\leq j\leq n\},$$
where $v_{i_1,j_1}v_{i_2,j_2}$ is an edge if and only if $|i_1-i_2|=1$ and $j_1=j_2$ or $i_1=i_2$ and $|j_1-j_2|=1.$ For convenience, let $f(S)=\sum_{v\in S}f(v)$ for any set $S$ of vertices. We say a vertex $v$ is {\it integer $\{2\}$-dominated} if $f(N[v])\geq 2.$ Also, an integer $\{2\}$-dominating function $f$ of $G_{m,n}$ is {\it minimal} if $f(V_{m,n})=\gamma_{\{2\}}(G_{m,n}).$ For a function $g: V_{m,n}\rightarrow \{0,1,2\}$, not necessarily an integer $\{2\}$-dominating function, we denote each vertex $v\in V_{m,n}$ by a circle filled with white, slashes, and black if $f(v)=0,$ $1$, and $2,$ respectively. For example, suppose that the grid $G_{3,2}$ has vertex set $V_{3,2}=\{v_{i,j}\mid\ 1\leq i\leq 3,1\leq j\leq2\}$ and a function $f:V_{3,2}\rightarrow \{0,1,2\}$ is defined by $$f(v_{i,j})=\begin{cases}
    2, \quad\text{if}~v_{i,j}=v_{3,1};\\
    1, \quad\text{if}~v_{i,j}=v_{1,2};\\
    0, \quad\text{otherwise.}
\end{cases}$$
Then, the mapped values and corresponding notations on its vertices are shown in Figure \ref{fig_notation}.

\begin{figure}[h]
\begin{center}
\begin{tikzpicture}
\draw (0,0)--(0,2);
\draw (1,0)--(1,2);
\draw (0,0)--(1,0);
\draw (0,1)--(1,1);
\draw (0,2)--(1,2);
\draw[fill] (0,0)coordinate(A00) circle(0.1cm);
\draw (-0.4,0) node {$v_{31}$};
\draw[fill=white] (1,0)coordinate(A10) circle(0.1cm);
\draw (1.4,0) node {$v_{32}$};
\draw[fill=white] (0,1)coordinate(A01) circle(0.1cm);
\draw (-0.4,1) node {$v_{21}$};
\draw[fill=white] (1,1)coordinate(A11) circle(0.1cm);
\draw (1.4,1) node {$v_{22}$};
\draw[fill=white] (0,2)coordinate(A02) circle(0.1cm);
\draw (-0.4,2) node {$v_{11}$};
\draw[fill=white] (1,2) circle(0.1cm);
\draw (1.4,2) node {$v_{12}$};
\draw[pattern=north east lines] (1,2)coordinate(A12) circle(0.1cm);
\end{tikzpicture}

\caption{A mapping $f$ and its corresponding notations on vertices of a $G_{3,2}.$}
\label{fig_notation}

\end{center}
\end{figure}
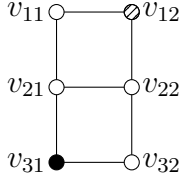

This paper is organized as follows. In Section~2, the exact values of $\gamma_{\{2\}}(G_{1,n})$ and $\gamma_{\{2\}}(G_{2,n})$ are confirmed. Additionally, an upper bound of $\gamma_{\{2\}}(G_{3,n})$ is given. In Section~3, we propose an algorithm to obtain $\gamma_{\{2\}}(G_{m,n})$ for arbitrary $m$ and $n$. The concluding remark is given in Section~4 and, as a demonstration of the algorithm, the exact values of $\gamma_{\{2\}}(G_{4,n})$ are conjectured.

\section{The values of $\gamma_{\{2\}}(G_{m,n})$ for $m\in\{1,2,3\}$}

The values of $\gamma_{\{2\}}(G_{1,n})$, $\gamma_{\{2\}}(G_{2,n})$, and $\gamma_{\{2\}}(G_{3,n})$ will be investigated in this section. More precisely, we first attain an upper bound of $\gamma_{\{2\}}(G_{m,n})$ by providing some dominating function for $m\in\{1,2,3\}$, and then verify that it is also a lower bound  for $m\in\{1,2\}$. Thus, the exact values of $\gamma_{\{2\}}(G_{1,n})$ and $\gamma_{\{2\}}(G_{2,n})$ are confirmed, while $\gamma_{\{2\}}(G_{3,n})$ has only an upper bound.

First of all, $\gamma_{\{2\}}(G_{1,n})$ is given in the following.
\begin{thm}     \label{thm_Gridmn_1}
For positive integer $n$, we have
$$\gamma_{\{2\}}(G_{1,n}) = 2  \left\lceil\frac{n}{3}\right\rceil.$$
\end{thm}
\begin{proof} Let $v_{1},v_{2},\dots,v_{n}$ be the vertices of a $G_{1,n}$, where $v_{i}v_{i+1}$ is an edge for $i=1,2,\dots,n-1$. Let $f$ be a minimal integer $\{2\}$-dominating function of $G_{1,n}$.
We show the result according to the remainder of $n$ modulo $3$.

\smallskip

\noindent{\bf Case 1.}
Assume $n=3k-2$ for some positive integer $k.$ If $k=1$ then the result is trivial. Suppose $k\geq 2.$ In order to dominate $v_1,$ $v_n$, and $v_{3i-2}$ for each $2\leq i\leq k-1$, we have
$$f(v_1)+f(v_2)\geq 2, \quad f(v_{n-1})+f(v_n)\geq 2,$$
and
$$f(v_{3i-3})+f(v_{3i-2})+f(v_{3i-1})\geq 2\quad \text{for each}~2\leq i\leq k-1,$$
respectively. Please see Figure \ref{fig_labeling_thmG1n_at.least2} for the partition of the vertices of $G_{1,n}.$
Hence,
$$\gamma_{\{2\}}(G_{1,n})=\sum_{i=1}^n f(v_i) \geq 2k=2\left\lceil\frac{n}{3}\right\rceil.$$

\begin{figure}[h]
\begin{center}
\begin{tikzpicture}
\draw (0,0)--(1,0);
\draw[dashed] (-0.25,-0.25) rectangle (1.25,0.4);
\draw (1,0)--(1.2,0);
\draw (1.8,0)--(2,0);
\draw (1.5,0) node{$\cdots$};
\draw (2,0)--(3,0);
\draw (3,0)--(4.2,0);
\draw[dashed] (1.75,-0.25) rectangle (4.25,0.4);
\draw (4.5,0) node{$\cdots$};
\draw (4.8,0)--(5,0);
\draw (5,0)--(6,0);
\draw[dashed] (4.75,-0.25) rectangle (6.25,0.4);
\draw[fill=white] (0,0)coordinate(A00) circle(0.1cm);
\draw (0,-0.4) node {$v_{1}$};
\draw[fill=white] (1,0)coordinate(A10) circle(0.1cm);
\draw (1,-0.4) node {$v_{2}$};
\draw[fill=white] (2,0)coordinate(A20) circle(0.1cm);
\draw (2,-0.4) node {$v_{3i-3}$};
\draw[fill=white] (3,0)coordinate(A30) circle(0.1cm);
\draw (3,-0.4) node {$v_{3i-2}$};
\draw[fill=white] (4,0)coordinate(A40) circle(0.1cm);
\draw (4,-0.4) node {$v_{3i-1}$};
\draw[fill=white] (5,0)coordinate(A50) circle(0.1cm);
\draw (5,-0.4) node {$v_{n-1}$};
\draw[fill=white] (6,0)coordinate(A60) circle(0.1cm);
\draw (6,-0.4) node {$v_{n}$};
\end{tikzpicture}

\caption{A partition of the vertices of $G_{1,n}$ for $n=3k-2.$} 
\label{fig_labeling_thmG1n_at.least2}

\end{center}
\end{figure}
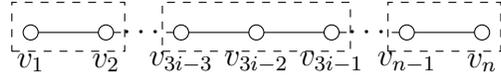

On the other hand, let $g$ be an integer $\{2\}$-dominating function of $G_{1,n}$ such that
$$g(v_{j})=\begin{cases}
2, \quad \text{if}~j\in \{1,n\}\cup\{3i-2\mid 2\leq i\leq k-1\};\\
0, \quad \text{otherwise}.
\end{cases}$$
Please refer to Figure \ref{fig_labeling_thmG1n} for the mapped values on $G_{1,n}$ by $g$, in which a vertex $v$ is filled with white and black if $g(v)=0$ and $2,$ respectively. Thus,
$$\gamma_{\{2\}}(G_{1,n})\leq \sum_{i=1}^n g(v_i) = 2k=2\left\lceil\frac{n}{3}\right\rceil.$$

\begin{figure}[h]
\begin{center}
\begin{tikzpicture}
\draw (0,0)--(1,0);
\draw (1,0)--(1.22,0);
\draw (1.78,0)--(2,0);
\draw (1.5,0) node{$\cdots$};
\draw (2,0)--(3,0);
\draw (3,0)--(4,0);
\draw (4.5,0) node{$\cdots$};
\draw (4,0)--(4.22,0);
\draw (4.78,0)--(5,0);
\draw (5,0)--(6,0);
\draw[fill] (0,0)coordinate(A00) circle(0.1cm);
\draw (0,-0.4) node {$v_{1}$};
\draw[fill=white] (1,0)coordinate(A10) circle(0.1cm);
\draw (1,-0.4) node {$v_{2}$};
\draw[fill=white] (2,0)coordinate(A20) circle(0.1cm);
\draw (2,-0.4) node {$v_{3i-3}$};
\draw[fill] (3,0)coordinate(A30) circle(0.1cm);
\draw (3,-0.4) node {$v_{3i-2}$};
\draw[fill=white] (4,0)coordinate(A40) circle(0.1cm);
\draw (4,-0.4) node {$v_{3i-1}$};
\draw[fill=white] (5,0)coordinate(A50) circle(0.1cm);
\draw (5,-0.4) node {$v_{n-1}$};
\draw[fill] (6,0)coordinate(A60) circle(0.1cm);
\draw (6,-0.4) node {$v_{n}$};
\end{tikzpicture}

\caption{An integer $\{2\}$-dominating function of $G_{1,n}$ for $n=3k-2.$} 
\label{fig_labeling_thmG1n}

\end{center}
\end{figure}
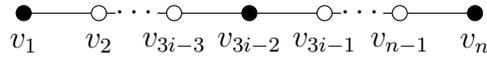

\noindent{\bf Case 2.}
Assume $n=3k-1$ for some positive integer $k.$ In order to dominate $v_1,$ and $v_{3i-2}$ for each $2\leq i\leq k$, we have $$f(v_1)+f(v_2)\geq 2,$$
and
$$f(v_{3i-3})+f(v_{3i-2})+f(v_{3i-1})\geq 2\quad \text{for each}~2\leq i\leq k,$$
respectively. Please see Figure \ref{fig_labeling_thmG1n_case2-1} for the partition of the vertices of $G_{1,n}.$ Hence,
$$\gamma_{\{2\}}(G_{1,n})=\sum_{i=1}^n f(v_i) \geq 2k=2\left\lceil\frac{n}{3}\right\rceil.$$

\begin{figure}[h]
\begin{center}
\begin{tikzpicture}
\draw (0,0)--(1,0);
\draw[dashed] (-0.25,-0.25) rectangle (1.25,0.4);
\draw (1,0)--(1.22,0);
\draw (1.78,0)--(2,0);
\draw (1.5,0) node{$\cdots$};
\draw (2,0)--(3,0);
\draw (3,0)--(4,0);
\draw[dashed] (1.75,-0.25) rectangle (4.25,0.4);
\draw (4.5,0) node{$\cdots$};
\draw (4,0)--(4.22,0);
\draw (4.78,0)--(5,0);
\draw[dashed] (4.725,-0.25) rectangle (7.25,0.4);
\draw (5,0)--(6,0);
\draw (6,0)--(7,0);
\draw[fill=white] (0,0)coordinate(A00) circle(0.1cm);
\draw (0,-0.4) node {$v_{1}$};
\draw[fill=white] (1,0)coordinate(A10) circle(0.1cm);
\draw (1,-0.4) node {$v_{2}$};
\draw[fill=white] (2,0)coordinate(A20) circle(0.1cm);
\draw (2,-0.4) node {$v_{3i-3}$};
\draw[fill=white] (3,0)coordinate(A30) circle(0.1cm);
\draw (3,-0.4) node {$v_{3i-2}$};
\draw[fill=white] (4,0)coordinate(A40) circle(0.1cm);
\draw (4,-0.4) node {$v_{3i-1}$};
\draw[fill=white] (5,0)coordinate(A50) circle(0.1cm);
\draw (5,-0.4) node {$v_{n-2}$};
\draw[fill=white] (6,0)coordinate(A60) circle(0.1cm);
\draw (6,-0.4) node {$v_{n-1}$};
\draw[fill=white] (7,0)coordinate(A70) circle(0.1cm);
\draw (7,-0.4) node {$v_{n}$};
\end{tikzpicture}

\caption{A partition of the vertices of $G_{1,n}$ for $n=3k-1.$} 
\label{fig_labeling_thmG1n_case2-1}

\end{center}
\end{figure}
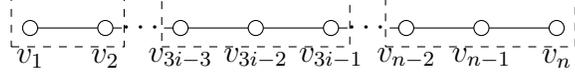

On the other hand, let $g$ be an integer $\{2\}$-dominating function of $G_{1,n}$ such that
$$g(v_{j})=\begin{cases}
2, \quad \text{if}~j\in \{1\}\cup\{3i-2\mid 2\leq i\leq k\};\\
0, \quad \text{otherwise}.
\end{cases}$$
Please refer to Figure \ref{fig_labeling_thmG1n_case2-2} for the mapped values on $G_{1,n}$ by $g$, in which a vertex $v$ is filled with white and black if $g(v)=0$ and $2,$ respectively. Thus,
$$\gamma_{\{2\}}(G_{1,n})\leq \sum_{i=1}^n g(v_i) = 2k=2\left\lceil\frac{n}{3}\right\rceil.$$

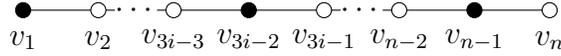
\begin{figure}[h]
\begin{center}
\begin{tikzpicture}
\draw (0,0)--(1,0);
\draw (1,0)--(1.22,0);
\draw (1.78,0)--(2,0);
\draw (1.5,0) node{$\cdots$};
\draw (2,0)--(3,0);
\draw (3,0)--(4,0);
\draw (4.5,0) node{$\cdots$};
\draw (4,0)--(4.22,0);
\draw (4.78,0)--(5,0);
\draw (5,0)--(6,0);
\draw (6,0)--(7,0);
\draw[fill=black] (0,0)coordinate(A00) circle(0.1cm);
\draw (0,-0.4) node {$v_{1}$};
\draw[fill=white] (1,0)coordinate(A10) circle(0.1cm);
\draw (1,-0.4) node {$v_{2}$};
\draw[fill=white] (2,0)coordinate(A20) circle(0.1cm);
\draw (2,-0.4) node {$v_{3i-3}$};
\draw[fill=black] (3,0)coordinate(A30) circle(0.1cm);
\draw (3,-0.4) node {$v_{3i-2}$};
\draw[fill=white] (4,0)coordinate(A40) circle(0.1cm);
\draw (4,-0.4) node {$v_{3i-1}$};
\draw[fill=white] (5,0)coordinate(A50) circle(0.1cm);
\draw (5,-0.4) node {$v_{n-2}$};
\draw[fill=black] (6,0)coordinate(A60) circle(0.1cm);
\draw (6,-0.4) node {$v_{n-1}$};
\draw[fill=white] (7,0)coordinate(A70) circle(0.1cm);
\draw (7,-0.4) node {$v_{n}$};
\end{tikzpicture}

\caption{An integer $\{2\}$-dominating function of $G_{1,n}$ for $n=3k-1.$} 
\label{fig_labeling_thmG1n_case2-2}

\end{center}
\end{figure}

\smallskip

\noindent{\bf Case 3.}
Assume $n=3k$ for some positive integer $k.$ For each $1\leq i\leq k,$ in order to dominate $v_{3i-1}$, we have
$$f(v_{3i-2})+f(v_{3i-1})+f(v_{3i})\geq 2.$$
Please see Figure \ref{fig_labeling_thmG1n_case3-1} for the partition of the vertices of $G_{1,n}.$
Hence,
$$\gamma_{\{2\}}(G_{1,n})=\sum_{i=1}^n f(v_i) \geq 2k=2\left\lceil\frac{n}{3}\right\rceil.$$

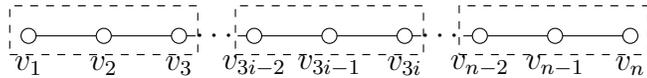
\begin{figure}[h]
\begin{center}
\begin{tikzpicture}
\draw (0,0)--(2,0);
\draw[dashed] (-0.25,-0.25) rectangle (2.25,0.4);
\draw (2,0)--(2.22,0);
\draw (2.78,0)--(3,0);
\draw (2.5,0) node{$\cdots$};
\draw (3,0)--(5,0);
\draw[dashed] (2.75,-0.25) rectangle (5.25,0.4);
\draw (5.5,0) node{$\cdots$};
\draw (5,0)--(5.22,0);
\draw (5.78,0)--(8,0);
\draw[dashed] (5.725,-0.25) rectangle (8.25,0.4);
\draw[fill=white] (0,0)coordinate(A00) circle(0.1cm);
\draw (0,-0.4) node {$v_{1}$};
\draw[fill=white] (1,0)coordinate(A10) circle(0.1cm);
\draw (1,-0.4) node {$v_{2}$};
\draw[fill=white] (2,0)coordinate(A20) circle(0.1cm);
\draw (2,-0.4) node {$v_{3}$};
\draw[fill=white] (3,0)coordinate(A30) circle(0.1cm);
\draw (3,-0.4) node {$v_{3i-2}$};
\draw[fill=white] (4,0)coordinate(A40) circle(0.1cm);
\draw (4,-0.4) node {$v_{3i-1}$};
\draw[fill=white] (5,0)coordinate(A50) circle(0.1cm);
\draw (5,-0.4) node {$v_{3i}$};
\draw[fill=white] (6,0)coordinate(A60) circle(0.1cm);
\draw (6,-0.4) node {$v_{n-2}$};
\draw[fill=white] (7,0)coordinate(A70) circle(0.1cm);
\draw (7,-0.4) node {$v_{n-1}$};
\draw[fill=white] (8,0)coordinate(A80) circle(0.1cm);
\draw (8,-0.4) node {$v_{n}$};
\end{tikzpicture}

\caption{A partition of the vertices of $G_{1,n},$ for $n=3k.$} 
\label{fig_labeling_thmG1n_case3-1}

\end{center}
\end{figure}

On the other hand, let $g$ be an integer $\{2\}$-dominating function of $G_{1,n}$ such that
$$g(v_{j})=\begin{cases}
2, \quad \text{if}~j\in \{3i-1\mid 1\leq i\leq k\};\\
0, \quad \text{otherwise}.
\end{cases}$$
Please refer to Figure \ref{fig_labeling_thmG1n_case3-2} for the mapped values on $G_{1,n}$ by $g$, in which a vertex $v$ is filled with white and black if $g(v)=0$ and $2,$ respectively. Thus,
$$\gamma_{\{2\}}(G_{1,n})\leq \sum_{i=1}^n g(v_i) = 2k = 2\left\lceil\frac{n}{3}\right\rceil.$$

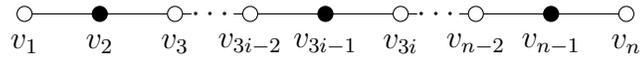
\begin{figure}[h]
\begin{center}
\begin{tikzpicture}
\draw (0,0)--(2,0);
\draw (2,0)--(2.22,0);
\draw (2.78,0)--(3,0);
\draw (2.5,0) node{$\cdots$};
\draw (3,0)--(5,0);
\draw (5.5,0) node{$\cdots$};
\draw (5,0)--(5.22,0);
\draw (5.78,0)--(8,0);
\draw[fill=white] (0,0)coordinate(A00) circle(0.1cm);
\draw (0,-0.4) node {$v_{1}$};
\draw[fill] (1,0)coordinate(A10) circle(0.1cm);
\draw (1,-0.4) node {$v_{2}$};
\draw[fill=white] (2,0)coordinate(A20) circle(0.1cm);
\draw (2,-0.4) node {$v_{3}$};
\draw[fill=white] (3,0)coordinate(A30) circle(0.1cm);
\draw (3,-0.4) node {$v_{3i-2}$};
\draw[fill] (4,0)coordinate(A40) circle(0.1cm);
\draw (4,-0.4) node {$v_{3i-1}$};
\draw[fill=white] (5,0)coordinate(A50) circle(0.1cm);
\draw (5,-0.4) node {$v_{3i}$};
\draw[fill=white] (6,0)coordinate(A60) circle(0.1cm);
\draw (6,-0.4) node {$v_{n-2}$};
\draw[fill] (7,0)coordinate(A70) circle(0.1cm);
\draw (7,-0.4) node {$v_{n-1}$};
\draw[fill=white] (8,0)coordinate(A80) circle(0.1cm);
\draw (8,-0.4) node {$v_{n}$};
\end{tikzpicture}

\caption{An integer $\{2\}$-dominating function of $G_{1,n},$ for $n=3k.$} 
\label{fig_labeling_thmG1n_case3-2}

\end{center}
\end{figure}

\end{proof}

\medskip

Before calculating the values of $\gamma_{\{2\}}(G_{2,n})$, we introduce an inspiring lemma.
\begin{lem}     \label{lem_max_degree}
Let $G=(V,E)$ be a graph with maximum degree $\Delta$ and $f$ be a minimal integer $\{2\}$-dominating function of $G$. Then
$$\gamma_{\{2\}}(G)=\sum_{v\in V}f(v)\geq\frac{2|V|}{\Delta+1}$$
with equality if and only if for every $v\in V$ we have $\sum_{u\in N[v]}f(u)=2,$ and for every vertex $w$ with $\deg(w)<\Delta$ we have $f(w)=0.$
\end{lem}
\begin{proof}
We use two-way counting on the value of
$$\sum_{u\in V}\sum_{v\in N[u]} f(u) = \sum_{u,v\in V:~uv\in E~\text{or}~u=v}f(u) = \sum_{v\in V}\sum_{u\in N[v]}f(u)$$
to show the result. Since $|N[u]|\leq \Delta+1$ for any $u\in V,$ we have
$$\sum_{u\in V}\sum_{v\in N[u]} f(u)\leq \sum_{u\in V}(\Delta+1)f(u) = \gamma_{\{2\}}(G)\cdot (\Delta+1)$$
with equality if and only if $f(u)=0$ for every vertex $u$ with $\deg(u)<\Delta$. On the other hand, since $\sum_{u\in N[v]}f(u)\geq 2$ for any $v\in V,$ we have
$$\sum_{v\in V}\sum_{u\in N[v]}f(u)\geq \sum_{v\in V}2 = 2|V|$$
with equality if and only if $\sum_{u\in N[v]}f(u)= 2$ for every $v\in V.$ The result follows.
\end{proof}

Then, the exact value of $\gamma_{\{2\}}(G_{2,n})$ is given in the following.
\begin{thm}     \label{thm_Gridmn_2}
For positive integer $n$, we have
$$\gamma_{\{2\}}(G_{2,n}) =  n+1.$$
\end{thm}
\begin{proof}
Let $\{v_{i,j}\mid  1\leq i\leq 2, 1\leq j\leq n\}$ be the vertex set of $G_{2,n}$, where $v_{i,j} v_{i,j+1}$ is an edge for each $1\leq i \leq 2$ and $1\leq j\leq n-1$, and $v_{1,\ell}v_{2,\ell}$ is an edge for each $1\leq\ell\leq n.$ Let $f$ be a minimal integer $\{2\}$-dominating function of $G_{2,n}.$ We show the result according to the parity of $n$.

\smallskip

\noindent{\bf Case 1.} Assume $n=2k-1$ for some positive integer $k.$ In order to dominate $v_{1,4j-3}$ for each $1\leq j\leq\left\lceil\frac{k}{2}\right\rceil$ and $v_{2,4\ell-1}$ for each $1\leq\ell\leq\left\lfloor\frac{k}{2}\right\rfloor.$ We have $$f(v_{1,4j-4})+f(v_{1,4j-3})+f(v_{1,4j-2})+f(v_{2,4j-3})\geq 2$$ for each $1\leq j\leq\left\lceil\frac{k}{2}\right\rceil,$ and $$f(v_{2,4\ell-2})+f(v_{2,4\ell-1})+f(v_{2,4\ell})+f(v_{1,4\ell-1})\geq 2$$ for each $1\leq\ell\leq\left\lfloor\frac{k}{2}\right\rfloor,$ respectively, where $f(v_{s,t}):=0$ if $t<1$ or $t>n.$ Thus, $$\gamma_{\{2\}}(G_{2,n})=\sum_{i=1}^2\sum_{j=1}^n f(v_{i,j})\geq2\left(\left\lceil\frac{k}{2}\right\rceil+\left\lfloor\frac{k}{2}\right\rfloor\right)=2k=n+1.$$
Please refer to Figure \ref{fig_labeling_thmG2n_case1-1} for the considered partition on the vertices of $G_{2,n}$.

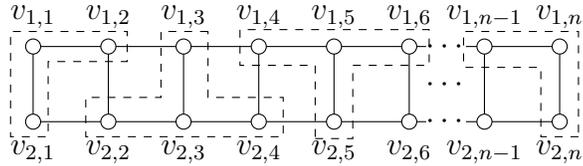
\begin{figure}[h]
\begin{center}
\begin{tikzpicture}


\draw[dashed]  (.2,.5)-|(.2,.8)-|(1.25,1.2)-|(-.3,-0.2)-|cycle;
\draw[dashed]  (3.75,-0.225)-|(4.25,0.75)-|(5.265,1.225)-|(2.75,0.75)-|cycle;
\draw[dashed]  (2.325,1.2)-|(1.7,0.3225)-|(0.7,-0.2)-|(3.325,0.3225)-|cycle;
\draw[dashed]  (7.25,-0.2)-|(6.75,0.5)-|(6.75,0.7225)-|(5.7225,1.2)-|cycle;

\draw (0,0)--(0,1);
\draw (1,0)--(1,1);
\draw (2,0)--(2,1);
\draw (3,0)--(3,1);
\draw (4,0)--(4,1);
\draw (5,0)--(5,1);
\draw (6,0)--(6,1);
\draw (7,0)--(7,1);
\draw (0,0)--(5.225,0);
\draw (5.755,0)--(7,0);
\draw (0,1)--(5.225,1);
\draw (5.755,1)--(7,1);

\draw (5.5,0) node{$\cdots$};
\draw (5.5,0.5) node{$\cdots$};
\draw (5.5,1) node{$\cdots$};

\draw (0,-0.4) node {$v_{2,1}$};
\draw (1,-0.4) node {$v_{2,2}$};
\draw (2,-0.4) node {$v_{2,3}$};
\draw (3,-0.4) node {$v_{2,4}$};
\draw (4,-0.4) node {$v_{2,5}$};
\draw (5,-0.4) node {$v_{2,6}$};
\draw (6,-0.4) node {$v_{2,n-1}$};
\draw (7,-0.4) node {$v_{2,n}$};
\draw (0,1.4) node {$v_{1,1}$};
\draw (1,1.4) node {$v_{1,2}$};
\draw (2,1.4) node {$v_{1,3}$};
\draw (3,1.4) node {$v_{1,4}$};
\draw (4,1.4) node {$v_{1,5}$};
\draw (5,1.4) node {$v_{1,6}$};
\draw (6,1.4) node {$v_{1,n-1}$};
\draw (7,1.4) node {$v_{1,n}$};
\draw[fill=white] (0,0)coordinate(A00) circle(0.1cm);
\draw[fill=white] (1,0)coordinate(A10) circle(0.1cm);
\draw[fill=white] (2,0)coordinate(A20) circle(0.1cm);
\draw[fill=white] (3,0)coordinate(A30) circle(0.1cm);
\draw[fill=white] (4,0)coordinate(A40) circle(0.1cm);
\draw[fill=white] (5,0)coordinate(A50) circle(0.1cm);
\draw[fill=white] (6,0)coordinate(A60) circle(0.1cm);
\draw[fill=white] (7,0)coordinate(A70) circle(0.1cm);
\draw[fill=white] (0,1)coordinate(A01) circle(0.1cm);
\draw[fill=white] (1,1)coordinate(A11) circle(0.1cm);
\draw[fill=white] (2,1)coordinate(A21) circle(0.1cm);
\draw[fill=white] (3,1)coordinate(A33) circle(0.1cm);
\draw[fill=white] (4,1)coordinate(A41) circle(0.1cm);
\draw[fill=white] (5,1)coordinate(A51) circle(0.1cm);
\draw[fill=white] (6,1)coordinate(A61) circle(0.1cm);
\draw[fill=white] (7,1)coordinate(A71) circle(0.1cm);
\end{tikzpicture}

\caption{A partition of the vertices of $G_{2,n},$ for $n=2k-1.$} 
\label{fig_labeling_thmG2n_case1-1}

\end{center}
\end{figure}

On the other hand, let $g$ be an integer $\{2\}$-dominating function of $G_{2,n}$ such that
$$g(v_{i,j})=\begin{cases}
1, \quad \text{if}~1\leq i\leq2~\text{and}~j=2\ell-1~\text{for}~1\leq\ell\leq k;\\
0, \quad\text{otherwise}.
\end{cases}$$
See Figure \ref{fig_labeling_thmG2n_case1-2} for an example of the function $g$ on a grid $G_{2,7},$ in which a vertex $v$ is filled with white and slashes if $g(v)=0$ and $1,$ respectively. Therefore,
$$\gamma_{\{2\}}(G_{2,n})\leq \sum_{i=1}^2 \sum_{j=1}^{2k} g(v_{i,j})=2k=n+1.$$

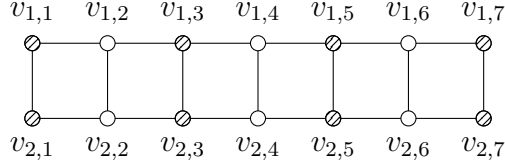
\begin{figure}[h]
\begin{center}
\begin{tikzpicture}

\draw (0,0)--(0,1);
\draw (6,0)--(6,1);
\draw (5,0)--(5,1);
\draw (4,0)--(4,1);
\draw (3,0)--(3,1);
\draw (2,0)--(2,1);
\draw (1,0)--(1,1);
\draw (0,0)--(6,0);
\draw (0,1)--(6,1);


\draw (0,-0.4) node {$v_{2,1}$};
\draw (1,-0.4) node {$v_{2,2}$};
\draw (2,-0.4) node {$v_{2,3}$};
\draw (3,-0.4) node {$v_{2,4}$};
\draw (4,-0.4) node {$v_{2,5}$};
\draw (5,-0.4) node {$v_{2,6}$};
\draw (6,-0.4) node {$v_{2,7}$};
\draw (0,1.4) node {$v_{1,1}$};
\draw (1,1.4) node {$v_{1,2}$};
\draw (2,1.4) node {$v_{1,3}$};
\draw (3,1.4) node {$v_{1,4}$};
\draw (4,1.4) node {$v_{1,5}$};
\draw (5,1.4) node {$v_{1,6}$};
\draw (6,1.4) node {$v_{1,7}$};
\draw[fill=white] (0,0)coordinate(A00) circle(0.1cm);
\draw[fill=white] (1,0)coordinate(A10) circle(0.1cm);
\draw[fill=white] (2,0)coordinate(A20) circle(0.1cm);
\draw[fill=white] (3,0)coordinate(A30) circle(0.1cm);
\draw[fill=white] (4,0)coordinate(A40) circle(0.1cm);
\draw[fill=white] (5,0)coordinate(A50) circle(0.1cm);
\draw[fill=white] (6,0)coordinate(A60) circle(0.1cm);
\draw[fill=white] (0,1)coordinate(A01) circle(0.1cm);
\draw[fill=white] (1,1)coordinate(A11) circle(0.1cm);
\draw[fill=white] (2,1)coordinate(A21) circle(0.1cm);
\draw[fill=white] (3,1)coordinate(A33) circle(0.1cm);
\draw[fill=white] (4,1)coordinate(A41) circle(0.1cm);
\draw[fill=white] (5,1)coordinate(A53) circle(0.1cm);
\draw[fill=white] (6,1)coordinate(A61) circle(0.1cm);
\draw[pattern=north east lines] (0,0)coordinate(A00) circle(0.1cm);
\draw[pattern=north east lines] (0,1)coordinate(A00) circle(0.1cm);
\draw[pattern=north east lines] (2,0)coordinate(A00) circle(0.1cm);
\draw[pattern=north east lines] (2,1)coordinate(A00) circle(0.1cm);
\draw[pattern=north east lines] (4,0)coordinate(A00) circle(0.1cm);
\draw[pattern=north east lines] (4,1)coordinate(A00) circle(0.1cm);
\draw[pattern=north east lines] (6,0)coordinate(A00) circle(0.1cm);
\draw[pattern=north east lines] (6,1)coordinate(A00) circle(0.1cm);
\end{tikzpicture}

\caption{An example of the function $g$ on a grid $G_{2,7}.$} 
\label{fig_labeling_thmG2n_case1-2}

\end{center}
\end{figure}

\smallskip

\noindent{\bf Case 2.}
Assume $n=2k$ for some positive integer $k.$ Since the maximum degree in $G_{2,n}$ is 3, by Lemma~\ref{lem_max_degree} we have
$$\gamma_{\{2\}}(G_{2,n})\geq\frac{2\cdot(4k)}{3+1}=2k.$$
To show $\gamma_{\{2\}}(G_{2,n})\geq2k+1,$ suppose to the contrary that $\gamma_{\{2\}}(G_{2,n})=2k.$ Then, the second part of Lemma~\ref{lem_max_degree} implies that the given minimal integer $\{2\}$-dominating function $f$ satisfies:
\begin{enumerate}
\item[(i)] $f(v_{1,1})=f(v_{2,1})=0$ since their degrees are less than the maximum degree $3,$
\item[(ii)] $f(v_{1,2})=f(v_{2,2})=2$ since $\sum_{v\in N[v_{1,1}]}f(v)=\sum_{v\in N[v_{2,1}]}f(v)=2,$ and
\item[(iii)] $\sum_{v\in N[v_{1,2}]}f(v)=2.$
\end{enumerate}
However, (ii) obviously contradicts to (iii) because
$$4=f(v_{1,2})+f(v_{2,2})\leq \sum_{v\in N[v_{1,2}]}f(v)=2.$$

On the other hand, to show $\gamma_{\{2\}}(G_{2,n})\leq 2k+1,$ let $g$ be an integer $\{2\}$-dominating function of $G_{2,n}$ such that $$g(v_{i,j})=\begin{cases}
    1, \quad \text{if}~1\leq i\leq2~\text{and}~j=2\ell-1~\text{for}~1\leq\ell\leq k;\\
    1, \quad \text{if}~i=1~\text{and}~j=2k;\\
    0, \quad \text{otherwise.}
\end{cases}$$
See Figure~\ref{fig_labeling_thmG2n_case2-2} for an example of the function $g$ on a grid $G_{2,6},$ in which a vertex $v$ is filled with white and slashes if $g(v)=0$ and $1$, respectively. Therefore,
$$\gamma_{2}(G_{2,n})\leq\sum_{i=1}^2\sum_{j=1}^n g(v_{i,j})=2k+1=n+1.$$

\begin{figure}[h]
\begin{center}
\begin{tikzpicture}

\draw (0,0)--(0,1);
\draw (2,0)--(2,1);
\draw (3,0)--(3,1);
\draw (4,0)--(4,1);
\draw (5,0)--(5,1);
\draw (0,0)--(5,0);
\draw (0,1)--(1,1);
\draw (1,1)--(2,1);
\draw (2,1)--(3,1);
\draw (3,1)--(4,1);
\draw (4,1)--(5,1);
\draw (1,1)--(1,0);


\draw (0,-0.4) node {$v_{2,1}$};
\draw (1,-0.4) node {$v_{2,2}$};
\draw (2,-0.4) node {$v_{2,3}$};
\draw (3,-0.4) node {$v_{2,4}$};
\draw (4,-0.4) node {$v_{2,5}$};
\draw (5,-0.4) node {$v_{2,6}$};
\draw (0,1.4) node {$v_{1,1}$};
\draw (1,1.4) node {$v_{1,2}$};
\draw (2,1.4) node {$v_{1,3}$};
\draw (3,1.4) node {$v_{1,4}$};
\draw (4,1.4) node {$v_{1,5}$};
\draw (5,1.4) node {$v_{1,6}$};
\draw[fill=white] (0,0)coordinate(A00) circle(0.1cm);
\draw[fill=white] (1,0)coordinate(A10) circle(0.1cm);
\draw[fill=white] (2,0)coordinate(A20) circle(0.1cm);
\draw[fill=white] (3,0)coordinate(A30) circle(0.1cm);
\draw[fill=white] (4,0)coordinate(A40) circle(0.1cm);
\draw[fill=white] (5,0)coordinate(A50) circle(0.1cm);
\draw[fill=white] (0,1)coordinate(A01) circle(0.1cm);
\draw[fill=white] (1,1)coordinate(A11) circle(0.1cm);
\draw[fill=white] (2,1)coordinate(A21) circle(0.1cm);
\draw[fill=white] (3,1)coordinate(A33) circle(0.1cm);
\draw[fill=white] (4,1)coordinate(A41) circle(0.1cm);
\draw[fill=white] (5,1)coordinate(A53) circle(0.1cm);
\draw[pattern=north east lines] (0,0)coordinate(A00) circle(0.1cm);
\draw[pattern=north east lines] (0,1)coordinate(A00) circle(0.1cm);
\draw[pattern=north east lines] (2,0)coordinate(A00) circle(0.1cm);
\draw[pattern=north east lines] (2,1)coordinate(A00) circle(0.1cm);
\draw[pattern=north east lines] (4,0)coordinate(A00) circle(0.1cm);
\draw[pattern=north east lines] (4,1)coordinate(A00) circle(0.1cm);
\draw[pattern=north east lines] (5,1)coordinate(A00) circle(0.1cm);
\end{tikzpicture}

\caption{An example of the function $g$ on a grid $G_{2,6}.$} 
\label{fig_labeling_thmG2n_case2-2}

\end{center}
\end{figure}
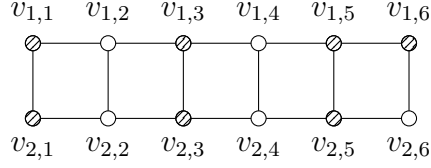

\end{proof}

\medskip

In the end of this section, an upper bound of $\gamma_{\{2\}}(G_{3,n})$ is proposed.
\begin{thm}     \label{thm_Gridmn_3}
For positive integer $n$, we have
$$\gamma_{\{2\}}(G_{3,n})\leq\left\lfloor\frac{3n}{2}\right\rfloor + 1.$$
\end{thm}
\begin{proof}
Let $\{v_{i,j}\mid1\leq i\leq3, 1\leq j\leq n\}$ be the vertex set of $G_{3,n}$, in which $v_{i1,j1}v_{i2,j2}$ is an edge if and only if $|i_{1}-i_{2}|=1$ and $j_{1}=j_{2}$ or $i_{1}=i_{2}$ and $|j_{1}-j_{2}|=1.$ To prove the result, We give an integer $\{2\}$-dominating function of $G_{3,n}$ with $\sum_{i=1}^3\sum_{j=1}^n f(v_{i,j})=\left\lfloor\frac{3n}{2}\right\rfloor+1.$ If $n=1,$ let $f(v_{2,1})=2$ and $f(v_{1,1})=f(v_{3,1})=0.$ If $n$ is even, let $$f(v_{i,j})=\begin{cases}
        1, \quad\text{if}~(i,j)=(2,1)~\text{or}~j~\text{is even;}\\
        0, \quad \text{otherwise.}
\end{cases}$$
One can see that $$\sum_{i=1}^3\sum_{j=1}^n~f(v_{i,j})=1+\frac{3}{2}n.$$
See Figure~\ref{fig_labeling_thmG3n_case1}
for the given $f$ on a grid $G_{3,n},$ in which a vertex $v$ is filled with white and slashes if $f(v)=0$ and $1,$ respectively.

\begin{figure}[h]
\begin{center}
\begin{tikzpicture}

\draw (0,2)--(2.225,2);
\draw (2.785,2)--(5,2);
\draw (2.525,1.55) node{$\cdots$};
\draw (2.55,2) node{$\cdots$};
\draw (0,1)--(2.225,1);
\draw (2.785,1)--(5,1);
\draw (2.55,1) node{$\cdots$};
\draw (2.525,0.55) node{$\cdots$};
\draw (0,0)--(2.225,0);
\draw (2.785,0)--(5,0);
\draw (2.55,0) node{$\cdots$};
\draw (0,0)--(0,2);
\draw (1,0)--(1,2);
\draw (2,0)--(2,2);
\draw (3,0)--(3,2);
\draw (4,0)--(4,2);
\draw (5,0)--(5,2);


\draw (0.3,2.2) node {$v_{1,1}$};
\draw (1.3,2.2) node {$v_{1,2}$};
\draw (2.3,2.2) node {$v_{1,3}$};
\draw (3.5,2.2) node {$v_{1,n-2}$};
\draw (4.5,2.2) node {$v_{1,n-1}$};
\draw (5.38,2.2) node {$v_{1,n}$};
\draw (0.3,0.2) node {$v_{3,1}$};
\draw (1.3,0.2) node {$v_{3,2}$};
\draw (2.3,0.2) node {$v_{3,3}$};
\draw (3.5,0.2) node {$v_{3,n-2}$};
\draw (4.5,0.2) node {$v_{3,n-1}$};
\draw (5.38,0.2) node {$v_{3,n}$};
\draw (0.3,1.2) node {$v_{2,1}$};
\draw (1.3,1.2) node {$v_{2,2}$};
\draw (2.3,1.2) node {$v_{2,3}$};
\draw (3.5,1.2) node {$v_{2,n-2}$};
\draw (4.5,1.2) node {$v_{2,n-1}$};
\draw (5.38,1.2) node {$v_{2,n}$};
\draw[fill=white] (0,0)coordinate(A00) circle(0.1cm);
\draw[fill=white] (0,2)coordinate(A02) circle(0.1cm);
\draw[fill=white] (1,2)coordinate(A12) circle(0.1cm);
\draw[fill=white] (2,2)coordinate(A22) circle(0.1cm);
\draw[fill=white] (3,2)coordinate(A32) circle(0.1cm);
\draw[fill=white] (4,2)coordinate(A42) circle(0.1cm);
\draw[fill=white] (5,2)coordinate(A32) circle(0.1cm);
\draw[fill=white] (1,0)coordinate(A10) circle(0.1cm);
\draw[fill=white] (2,0)coordinate(A20) circle(0.1cm);
\draw[fill=white] (3,0)coordinate(A30) circle(0.1cm);
\draw[fill=white] (4,0)coordinate(A40) circle(0.1cm);
\draw[fill=white] (5,0)coordinate(A30) circle(0.1cm);
\draw[fill=white] (0,1)coordinate(A01) circle(0.1cm);
\draw[fill=white] (1,1)coordinate(A11) circle(0.1cm);
\draw[fill=white] (2,1)coordinate(A21) circle(0.1cm);
\draw[fill=white] (3,1)coordinate(A33) circle(0.1cm);
\draw[fill=white] (4,1)coordinate(A41) circle(0.1cm);
\draw[fill=white] (5,1)coordinate(A33) circle(0.1cm);
\draw[pattern=north east lines] (0,1)coordinate(A01) circle(0.1cm);
\draw[pattern=north east lines] (1,0)coordinate(A01) circle(0.1cm);
\draw[pattern=north east lines] (1,1)coordinate(A01) circle(0.1cm);
\draw[pattern=north east lines] (1,2)coordinate(A01) circle(0.1cm);
\draw[pattern=north east lines] (3,0)coordinate(A01) circle(0.1cm);
\draw[pattern=north east lines] (3,1)coordinate(A01) circle(0.1cm);
\draw[pattern=north east lines] (3,2)coordinate(A01) circle(0.1cm);
\draw[pattern=north east lines] (5,0)coordinate(A01) circle(0.1cm);
\draw[pattern=north east lines] (5,1)coordinate(A01) circle(0.1cm);
\draw[pattern=north east lines] (5,2)coordinate(A01) circle(0.1cm);
\end{tikzpicture}

\caption{An integer $\{2\}$-dominating function $f$ on a grid $G_{3,n}$ for even $n.$} 
\label{fig_labeling_thmG3n_case1}

\end{center}
\end{figure}
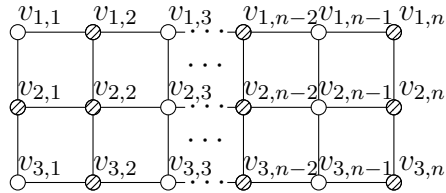

If $n>1$ is odd, let $$f(v_{i,j})=\begin{cases}
    1, \quad \text{if}~(i,j)=(2,1)~\text{or}~(2,n)~\text{or}~j~\text{ is even};\\
    0, \quad \text{otherwise.}\\
\end{cases}$$
One can see that
$$\sum_{i=1}^3\sum_{j=1}^n f(v_{i,j})=1+1+\frac{n-1}{2}\cdot3=1+\left\lfloor\frac{3n}{2}\right\rfloor.$$
See Figure~\ref{fig_labeling_thmG3n_case2} for the given $f$ on a grid $G_{3,n},$ in which a vertex $v$ is filled with white and slashes if $f(v)=0$ and $1,$ respectively. The proof is completed.

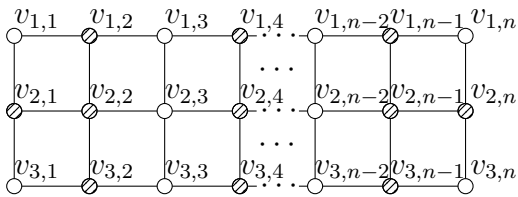
\begin{figure}[h]
\begin{center}
\begin{tikzpicture}

\draw (0,2)--(3.225,2);
\draw (3.785,2)--(6,2);
\draw (3.525,1.55) node{$\cdots$};
\draw (3.55,2) node{$\cdots$};
\draw (0,1)--(3.2225,1);
\draw (3.785,1)--(6,1);
\draw (3.55,1) node{$\cdots$};
\draw (3.525,0.55) node{$\cdots$};
\draw (0,0)--(3.225,0);
\draw (3.785,0)--(6,0);
\draw (3.55,0) node{$\cdots$};
\draw (0,0)--(0,2);
\draw (1,0)--(1,2);
\draw (2,0)--(2,2);
\draw (3,0)--(3,2);
\draw (4,0)--(4,2);
\draw (5,0)--(5,2);
\draw (6,0)--(6,2);

\draw (0.3,2.2) node {$v_{1,1}$};
\draw (1.3,2.2) node {$v_{1,2}$};
\draw (2.3,2.2) node {$v_{1,3}$};
\draw (3.3,2.2) node {$v_{1,4}$};
\draw (4.5,2.2) node {$v_{1,n-2}$};
\draw (5.5,2.2) node {$v_{1,n-1}$};
\draw (6.38,2.2) node {$v_{1,n}$};
\draw (0.3,0.2) node {$v_{3,1}$};
\draw (1.3,0.2) node {$v_{3,2}$};
\draw (2.3,0.2) node {$v_{3,3}$};
\draw (3.3,0.2) node {$v_{3,4}$};
\draw (4.5,0.2) node {$v_{3,n-2}$};
\draw (5.5,0.2) node {$v_{3,n-1}$};
\draw (6.38,0.2) node {$v_{3,n}$};
\draw (0.3,1.2) node {$v_{2,1}$};
\draw (1.3,1.2) node {$v_{2,2}$};
\draw (2.3,1.2) node {$v_{2,3}$};
\draw (3.3,1.2) node {$v_{2,4}$};
\draw (4.5,1.2) node {$v_{2,n-2}$};
\draw (5.5,1.2) node {$v_{2,n-1}$};
\draw (6.38,1.2) node {$v_{2,n}$};
\draw[fill=white] (0,0)coordinate(A00) circle(0.1cm);
\draw[fill=white] (0,1)coordinate(A01) circle(0.1cm);
\draw[fill=white] (0,2)coordinate(A02) circle(0.1cm);
\draw[fill=white] (1,2)coordinate(A12) circle(0.1cm);
\draw[fill=white] (2,2)coordinate(A22) circle(0.1cm);
\draw[fill=white] (3,2)coordinate(A32) circle(0.1cm);
\draw[fill=white] (4,2)coordinate(A42) circle(0.1cm);
\draw[fill=white] (5,2)coordinate(A52) circle(0.1cm);
\draw[fill=white] (6,2)coordinate(A62) circle(0.1cm);
\draw[fill=white] (1,0)coordinate(A10) circle(0.1cm);
\draw[fill=white] (2,0)coordinate(A20) circle(0.1cm);
\draw[fill=white] (3,0)coordinate(A30) circle(0.1cm);
\draw[fill=white] (4,0)coordinate(A40) circle(0.1cm);
\draw[fill=white] (5,0)coordinate(A50) circle(0.1cm);
\draw[fill=white] (6,0)coordinate(A60) circle(0.1cm);
\draw[fill=white] (1,1)coordinate(A11) circle(0.1cm);
\draw[fill=white] (2,1)coordinate(A21) circle(0.1cm);
\draw[fill=white] (3,1)coordinate(A31) circle(0.1cm);
\draw[fill=white] (4,1)coordinate(A41) circle(0.1cm);
\draw[fill=white] (5,1)coordinate(A51) circle(0.1cm);
\draw[fill=white] (6,1)coordinate(A61) circle(0.1cm);
\draw[pattern=north east lines] (0,1)coordinate(A01) circle(0.1cm);
\draw[pattern=north east lines] (1,0)coordinate(A01) circle(0.1cm);
\draw[pattern=north east lines] (1,1)coordinate(A01) circle(0.1cm);
\draw[pattern=north east lines] (1,2)coordinate(A01) circle(0.1cm);
\draw[pattern=north east lines] (3,0)coordinate(A01) circle(0.1cm);
\draw[pattern=north east lines] (3,1)coordinate(A01) circle(0.1cm);
\draw[pattern=north east lines] (3,2)coordinate(A01) circle(0.1cm);
\draw[pattern=north east lines] (5,0)coordinate(A01) circle(0.1cm);
\draw[pattern=north east lines] (5,1)coordinate(A01) circle(0.1cm);
\draw[pattern=north east lines] (5,2)coordinate(A01) circle(0.1cm);
\draw[pattern=north east lines] (6,1)coordinate(A01) circle(0.1cm);
\end{tikzpicture}

\caption{An example of the map the values on a grid $G_{3,n}$ for odd $n.$} 
\label{fig_labeling_thmG3n_case2}

\end{center}
\end{figure}
\end{proof}

We believe that the upper bound for $\gamma_{\{2\}}(G_{3,n})$ in Theorem~\ref{thm_Gridmn_3} is also a lower bound.
\begin{conj}  \label{conj_gamma2G3n}
For positive integer $n$, we have
$$\gamma_{\{2\}}(G_{3,n})=\left\lfloor\frac{3n}{2}\right\rfloor + 1.$$
\end{conj}

\smallskip

We give some comment on the exact value of $\gamma_{\{2\}}(G_{3,n})$.
\begin{rem}
In fact, by an algorithm we propose in the next section, the exact value of $\gamma_{\{2\}}(G_{3,n})$ is $\left\lfloor\frac{3n}{2}\right\rfloor+1$ for any input positive integer $n$. However, it is not sufficient for us to assert that $\gamma_{\{2\}}(G_{3,n})=\left\lfloor\frac{3n}{2}\right\rfloor+1$. As a future work, we suggest that the methods proposed in~\cite{gprt11} and~\cite{rt19} might be modified and applied on this problem.
\end{rem}

\medskip

\section{An algorithm for determining $\gamma_{\{2\}}(G_{m,n})$}

In this section, an algorithm will be given to count the value of $\gamma_{\{2\}}(G_{m,n})$ for arbitrary positive integers $m$ and $n$. First of all, we introduce a method of labeling on the vertices of a grid $G_{m,n}$ with vertex set $$V_{m,n}=\{v_{i,j}~\mid~ 1\leq i\leq m, 1\leq j\leq n\},$$
in which $v_{i_1,j_1}v_{i_2,j_2}$ is an edge if and only if $|i_1-i_2|=1$ and $j_1=j_2$ or $i_1=i_2$ and $|j_1-j_2|=1.$ Let $f:V_{m,n}\rightarrow \{0,1,2\}$ be a function, and each vertex $v\in V_{m,n}$ be labeled with
$$\left\{
\begin{array}{ll}
\alpha_2, & \text{if}\quad f(v)=2;
\\
\alpha_{11}, & \text{if}\quad f(v)=1~\text{and}~f(N(v))\geq 1;
\\
\alpha_{10}, & \text{if}\quad f(v)=1~\text{and}~f(N(v))=0;
\\
\alpha_{02}, & \text{if}\quad f(v)=0~\text{and}~f(N(v))\geq 2;
\\
\alpha_{01}, & \text{if}\quad f(v)=0~\text{and}~f(N(v))=1;
\\
\alpha_{00}, & \text{if}\quad f(v)=0~\text{and}~f(N(v))=0.
\end{array}
\right.$$

We may realize the labeling $\alpha_{i,j}$ on a vertex $v$ by placing $i$ stones on $v$ and (exactly or at least) $j$ stones on the neighbors of $v$. Recall that for a function $f:V_{m,n}\rightarrow \{0,1,2\}$ we denote each vertex $v\in V_{m,n}$ by a circle filled with white, slashes, and black if $f(v)=0,$ $1$, and $2$, respectively. In the example given in Figure~\ref{fig_notation}, we define a function $f:V_{3,2}\rightarrow \{0,1,2\}$ by
$$f(v_{i,j})=\begin{cases}
    2, \quad\text{if}~(i,j)=(3,1);\\
    1, \quad\text{if}~(i,j)=(1,2);\\
    0, \quad\text{otherwise.}
\end{cases}$$
Then, its labeling corresponding to $f$ is shown in Figure~\ref{fig_labeling}.

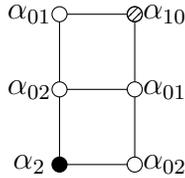
\begin{figure}[h]
\begin{center}
\begin{tikzpicture}
\draw (0,0)--(0,2);
\draw (1,0)--(1,2);
\draw (0,0)--(1,0);
\draw (0,1)--(1,1);
\draw (0,2)--(1,2);
\draw[fill] (0,0)coordinate(A00) circle(0.1cm);
\draw (-0.4,0) node {$\alpha_2$};
\draw[fill=white] (1,0)coordinate(A10) circle(0.1cm);
\draw (1.4,0) node {$\alpha_{02}$};
\draw[fill=white] (0,1)coordinate(A01) circle(0.1cm);
\draw (-0.4,1) node {$\alpha_{02}$};
\draw[fill=white] (1,1)coordinate(A11) circle(0.1cm);
\draw (1.4,1) node {$\alpha_{01}$};
\draw[fill=white] (0,2)coordinate(A02) circle(0.1cm);
\draw (-0.4,2) node {$\alpha_{01}$};
\draw[fill=white] (1,2) circle(0.1cm);
\draw (1.4,2) node {$\alpha_{10}$};
\draw[pattern=north east lines] (1,2)coordinate(A12) circle(0.1cm);
\end{tikzpicture}

\caption{An example of labeling on a $G_{3,2}$ equipped with a function $f.$}
\label{fig_labeling}

\end{center}
\end{figure}

By considering the mapped values of each vertex and its neighbors, the following labeling of pairs and triples of consecutive entries are prohibited:
\begin{enumerate}
\item[(i)] A pair of consecutive vertices is labeled by $(\alpha_2,\alpha_{10}),$ $(\alpha_2,\alpha_{01}),$ $(\alpha_2,\alpha_{00}),$ $(\alpha_{11},\alpha_{10}),$ $(\alpha_{11},\alpha_{00}),$ $(\alpha_{10},\alpha_2),$ $(\alpha_{10},\alpha_{11}),$ $(\alpha_{10},\alpha_{10}),$ $(\alpha_{10},\alpha_{00}),$ $(\alpha_{01},\alpha_2),$ $(\alpha_{00},\alpha_2),$ $(\alpha_{00},\alpha_{11}),$ or $(\alpha_{00},\alpha_{10}).$
\item[(ii)] A triple of consecutive vertices is labeled by $(\alpha_{11},\alpha_{01},\alpha_{11}),$ $(\alpha_{11},\alpha_{01},\alpha_{10}),$ $(\alpha_{10},\alpha_{01},\alpha_{11}),$ or $(\alpha_{10},\alpha_{01},\alpha_{10}).$
\end{enumerate}
It is straightforward to see their prohibition. For example, if a vertex $v$ is labeled by $\alpha_2$, then $f(v)=2$ and every vertex $u$ adjacent to $v$ satisfies $f(N(u))\geq 2,$ where $f$ is the corresponding function. Consequently, the labeling of two consecutive vertices will never be $(\alpha_2,\alpha_{10})$, $(\alpha_2,\alpha_{01}),$ or $(\alpha_2,\alpha_{00})$.

\smallskip

We say that a column vector with entries labeled with $\alpha_{2},$ $\alpha_{11},$ $\alpha_{10},$ $\alpha_{02},$ $\alpha_{01},$ $\alpha_{00}$ is {\it feasible,} if all the prohibited patterns do not appear. However, we emphasize that a column vector must be labeled after considering its neighborhood. For example, the grid $G_{3,4}$ in Figure~\ref{fig_u3_example_thmG341_case1} equipped with a function $f:V_{3,4}\rightarrow \{0,1,2\}$ as shown by three different patterns on each vertex. Then, the labeling $(\alpha_{10},\alpha_{02},\alpha_{00})^T$ on the last (rightmost) column is feasible since it does not contain any prohibited pattern. Let $T(m)$ be the set of feasible column vectors of length $m.$ By programming, we have $|T(1)|=6,$ $|T(2)|=23,$ $|T(3)|=95,$ $|T(4)|=389,$ $|T(5)|=1595,$ $|T(6)|=6538,$ and $|T(7)|=26802.$ Currently, we do not see a recurrence relation.

Next, assume that $m$ is fixed. We introduce a class ${\cal{F}}_{n}$ of functions in which each function almost forms an integer $\{2\}$-dominating function on $G_{m,n}$ if we neglect the last column.
\begin{defn}     \label{defi_F_n}
Given $m$, let ${\cal{F}}_{n}$ be the set consisting of functions $f:V_{m,n}\rightarrow\{0,1,2\}$ such that $\sum_{u\in N[v]} f(u)\geq2$ for all vertices $v$ not in the last column of $G_{m,n}.$ That is, every vertex not in the last column of $G_{m,n}$ will be labeled with $\alpha_2, \alpha_{11},$ or $\alpha_{02}$ if $G_{m,n}$ is equipped with any $f\in{\cal{F}}_{n}.$ 
\end{defn}

We do not restrict the labeling of vertices on the last column. For example, if $m=3,$ $n=4,$ and the function $f:V_{m,n}\rightarrow\{0,1,2\}$ given in Figure \ref{fig_u3_example_thmG341_case1} is in ${\cal{F}}_{4}$ that may have labels $\alpha_{10},$ $\alpha_{01},$ or $\alpha_{00}$ in the last column.

\begin{figure}[h]
\begin{center}
\begin{tikzpicture}

\draw (0,0)--(3,0);
\draw (0,1)--(3,1);
\draw (0,2)--(3,2);

\draw (0,0)--(0,2);
\draw (1,0)--(1,2);
\draw (2,0)--(2,2);
\draw (3,0)--(3,2);

\draw (1.38,0.15) node {$\alpha_{2}$};
\draw (1.38,1.15) node {$\alpha_{2}$};
\draw (1.38,2.15) node {$\alpha_{2}$};
\draw (0.38,2.15) node {$\alpha_{02}$};
\draw (0.38,1.15) node {$\alpha_{11}$};
\draw (0.38,0.15) node {$\alpha_{02}$};
\draw (2.38,2.15) node {$\alpha_{02}$};
\draw (2.38,1.15) node {$\alpha_{11}$};
\draw (2.38,0.15) node {$\alpha_{02}$};
\draw (3.38,2.15) node {$\alpha_{10}$};
\draw (3.38,1.15) node {$\alpha_{02}$};
\draw (3.38,0.15) node {$\alpha_{00}$};

\draw[fill] (1,1)coordinate(A11) circle(0.1cm);
\draw[fill=white] (3,1)coordinate(A33) circle(0.1cm);
\draw[fill=white] (0,0)coordinate(A00) circle(0.1cm);
\draw[fill=white] (0,1)coordinate(A01) circle(0.1cm);
\draw[fill=white] (0,2)coordinate(A02) circle(0.1cm);
\draw[fill] (1,2)coordinate(A12) circle(0.1cm);
\draw[fill=white] (2,2)coordinate(A22) circle(0.1cm);
\draw[fill=white] (3,2)coordinate(A32) circle(0.1cm);
\draw[pattern=north east lines] (3,2)coordinate(A32) circle(0.1cm);
\draw[fill=white] (2,1)coordinate(A21) circle(0.1cm);
\draw[pattern=north east lines] (2,1)coordinate(A21) circle(0.1cm);
\draw[fill] (1,0)coordinate(A10) circle(0.1cm);
\draw[fill=white] (2,0)coordinate(A20) circle(0.1cm);
\draw[fill=white] (3,0)coordinate(A30) circle(0.1cm);
\draw[pattern=north east lines] (0,1)coordinate(A01) circle(0.1cm);
\end{tikzpicture}

\caption{A function $f\in{\cal{F}}_{4}$ and its corresponding labeling.} 
\label{fig_u3_example_thmG341_case1}

\end{center}
\end{figure}
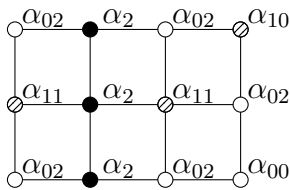

Before depicting our algorithm, we define {\it weights} in a natural way.
\begin{defn}  \label{defn_FnX}
Assume that $m$ is a positive integer, and $X\in T(m)$ is a feasible labeled column vector.
Let ${\cal{F}}_{n}(X)$ be the set of the functions in ${\cal{F}}_{n}$ that having the last column labeled by $X.$ Suppose $f:V_{m,n}\rightarrow\{0,1,2\}$ is a function in ${\cal{F}}_{n}(X)$. Then,
The {\it weight} of a function $f:V_{m,n}\rightarrow\{0,1,2\}$ is defined by
$$|f|=\sum_{i=1}^m\sum_{j=1}^n f(v_{i,j}),$$
the {\it weight} of a labeled column vector $X$ is defined by
$$|X|=\sum_{i=1}^m s_i,\quad \text{where}~s_i=\left\{\begin{array}{ll}
2 & \text{if}~X_j(k)=\alpha_2;
\\
1 & \text{if}~X_j(k)\in\{\alpha_{11},\alpha_{10}\};
\\
0 & \text{if}~X_j(k)\in\{\alpha_{02},\alpha_{01},\alpha_{00}\};
\end{array}\right. $$
and the {\it weight} of a grid $G_{m,n}$ equipped with functions in ${\cal{F}}_{n}(X)$ is defined by
$${{\boldsymbol{w}}_{n}}(X)=\min \{|f|:f\in{\cal{F}}_{n}(X)\}.$$
\end{defn}

To show that Definition~\ref{defn_FnX} is well-defined, we give the following proposition.
\begin{prop}     \label{prop_FnX}
Every $X\in T(m)$ can be the labeling on the last column of $G_{m,n}$ corresponding to some $f\in{\cal{F}}_{n}(X)$, if $n\geq 3.$
That is, ${\cal{F}}_{n}(X)$ is not empty.
\end{prop}
\begin{proof}
Let $f:V_{m,n}\rightarrow \{0,1,2\}$ be a function with $f(v_{i,j})=2$ if $1\leq j\leq n-2.$ For the last two columns, we first place number of stones on each vertex of the last column that fit the first coordinates of entries of $X$, then place the corresponding number of stones on each vertex of the penultimate column that fit the second coordinates of entries of $X$. It is clear that $f\in{\cal{F}}_{n}(X).$
\end{proof}

We give an example to Proposition~\ref{prop_FnX}.
\begin{exam}
Let $m=4,$ $n=4,$ and $X=(\alpha_{11},\alpha_{02},\alpha_{01},\alpha_{10})^T\in T(m).$ Then a function $f:V_{m,n}\rightarrow\{0,1,2\}$ described in Proposition~\ref{prop_FnX} is shown in Figure~\ref{fig_example_FnX}. Notice that if $f(v_{1,3})$ or $f(v_{2,3})$ is changed to $2$ then $f$ is still an example.
\end{exam}

\begin{figure}[h]
\begin{center}
\begin{tikzpicture}

\draw (0,0)--(3,0);
\draw (0,1)--(3,1);
\draw (0,2)--(3,2);
\draw (0,3)--(3,3);
\draw (0,0)--(0,3);
\draw (1,0)--(1,3);
\draw (2,0)--(2,3);
\draw (3,0)--(3,3);

\draw (3.4,3) node {$\alpha_{11}$};
\draw (3.4,2) node {$\alpha_{02}$};
\draw (3.4,1) node {$\alpha_{01}$};
\draw (3.4,0) node {$\alpha_{10}$};
\draw (2.4,2.6) node {$v_{1,3}$};
\draw (2.4,1.6) node {$v_{2,3}$};

\draw[fill=black] (0,0)coordinate(A00) circle(0.1cm);
\draw[fill=black] (0,1)coordinate(A01) circle(0.1cm);
\draw[fill=black] (0,2)coordinate(A02) circle(0.1cm);
\draw[fill=black] (0,3)coordinate(A03) circle(0.1cm);
\draw[fill=black] (1,0)coordinate(A10) circle(0.1cm);
\draw[fill=black] (1,1)coordinate(A11) circle(0.1cm);
\draw[fill=black] (1,2)coordinate(A12) circle(0.1cm);
\draw[fill=black] (1,3)coordinate(A13) circle(0.1cm);
\draw[fill=white] (2,0)coordinate(A20) circle(0.1cm);
\draw[fill=white] (2,1)coordinate(A21) circle(0.1cm);
\draw[fill=white] (2,2)coordinate(A22) circle(0.1cm);
\draw[pattern=north east lines] (2,2)coordinate(A22) circle(0.1cm);
\draw[fill=white] (2,3)coordinate(A23) circle(0.1cm);
\draw[pattern=north east lines] (2,3)coordinate(A23) circle(0.1cm);
\draw[fill=white] (3,0)coordinate(A30) circle(0.1cm);
\draw[pattern=north east lines] (3,0)coordinate(A30) circle(0.1cm);
\draw[fill=white] (3,1)coordinate(A31) circle(0.1cm);
\draw[fill=white] (3,2)coordinate(A32) circle(0.1cm);
\draw[fill=white] (3,3)coordinate(A33) circle(0.1cm);
\draw[pattern=north east lines] (3,3)coordinate(A33) circle(0.1cm);

\draw[fill=white] (3,1)coordinate(A33) circle(0.1cm);

\end{tikzpicture}

\caption{A function $f\in{\cal{F}}_{4}(X)$.} 
\label{fig_example_FnX}

\end{center}
\end{figure}
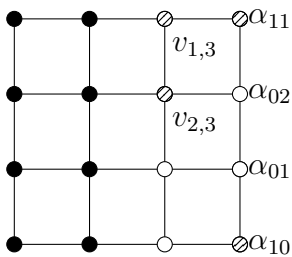

We then provide an algorithm to find $\gamma_{\{2\}}(G_{m,n})$. First of all, the relation between ${{\boldsymbol{w}}_{n}}(X)$ and $\gamma_{\{2\}}(G_{m,n})$ must be verified.
\begin{lem}
For given $m$, the integer $\{2\}$-domination number of $G_{m,n}$ is
$$\gamma_{\{2\}}(G_{m,n}) = \min\left\{{{\boldsymbol{w}}_{n}}(X)\mid X\in T^D(m)\right\}$$
where $T^D(m)\subseteq T(m)$ collects all feasible column vectors of length $m$ with labels in $\{\alpha_2,\alpha_{11},\alpha_{02}\}$ only (and without labels in $\{\alpha_{10},\alpha_{01},\alpha_{00}\}$).
\end{lem}
\begin{proof}
If $f\in{\cal{F}}_{n}(X)$ for some labeled column vector $X\in T^D(m),$ then $X$ only has labels in $\{\alpha_2,\alpha_{11},\alpha_{02}\}$ so that $f$ is an integer $\{2\}$-dominating function on $G_{m,n}.$ Thus, $\gamma_{\{2\}}(G_{m,n})\leq\min\left\{{{\boldsymbol{w}}_{n}}(X)\mid X\in T^D(m)\right\}.$
On the other hand, if $g$ is a minimal integer $\{2\}$-dominating function on $G_{m,n}$, then the labeling on the last column of $G_{m,n}$ corresponding to $g$ must be a feasible column vector with labels in $\{\alpha_2,\alpha_{11},\alpha_{02}\}$ only. Hence,
$$\gamma_{\{2\}}(G_{m,n})=|g|\geq\min\left\{{{\boldsymbol{w}}_{n}}(X)\mid X\in T^D(m)\right\}.$$
The result follows.
\end{proof}

By the above Lemma, we shall give an algorithm to find $\min\left\{{{\boldsymbol{w}}_{n}}(X)\mid X\in T^D(m)\right\}$, or equivalently, $\gamma_{\{2\}}(G_{m,n})$, in the following. Furthermore, if $m$ is prescribed, then the time complexity of this algorithm is linear in $n$.

\medskip

\noindent{{\bf Algorithm for determining $\gamma_{\{2\}}(G_{m,n})$ for fixed $m$}}
\begin{enumerate}
\item List the feasible column vectors as $X_1, X_2, \ldots, X_\ell$, where $\ell=|T(m)|.$
\item Construct a truth table $M$ of size $\ell$-by-$\ell$ in which $M(i,j)$ is True if and only if there is a function
$$f:V_{m,n-1}\rightarrow\{0,1,2\}\quad\text{with}\quad f\in{\cal{F}}_{n-1}(X_i)$$
such that the extended function
$$\tilde{f}:V_{m,n}\rightarrow\{0,1,2\}$$
obtained from $f$ by concatenating a new column
$$\tilde{f}(v_{k,n})=\left\{\begin{array}{ll}
2 & \text{if}~X_j(k)=\alpha_2;
\\
1 & \text{if}~X_j(k)\in\{\alpha_{11},\alpha_{10}\};
\\
0 & \text{if}~X_j(k)\in\{\alpha_{02},\alpha_{01},\alpha_{00}\};
\end{array}\right.
\quad\quad k=1,2,\ldots,m$$
satisfies $\tilde{f}\in {\cal{F}}_{n}(X_j).$
\item Initialize an array $w$ of length $\ell$ such that
    $$w(i) = {{\boldsymbol{w}}_{0}}(X_i) := \left\{\begin{array}{ll}
0 & \text{if}~X_i=(\alpha_{02},\alpha_{02},\ldots,\alpha_{02})^T;
\\
\infty & \text{otherwise}.
\end{array}\right.$$
\item For $k$ from $1$ to $n$ do:

    \quad Replace the value of $w(j)$ by
    $${{\boldsymbol{w}}_{k}}(X_j) = \min\{{{\boldsymbol{w}}_{k-1}}(X_i)\mid M(i,j)~\text{is True}\}+|X_j|$$
    \quad for all $1\leq j\leq \ell,$ where $|X_j|$ is defined in Definition~\ref{defn_FnX}.
\item Return
    $$\min\{w(i)\mid X_i\in T^D(m)\}=\min\left\{{{\boldsymbol{w}}_{n}}(X_i)\mid X_i\in T^D(m)\right\}=\gamma_{\{2\}}(G_{m,n}).$$
\end{enumerate}

We give more details on the truth table $M$. For example, let $m=4,$ $n\geq 3,$ and $X_j=(\alpha_{11},\alpha_{02},\alpha_{01},\alpha_{10})^T$. Recall that $M(i,j)$ is True if and only if there is a function $f\in{\cal{F}}_{n-1}(X_i)$ such that the last column is labeled by $X_i$, and for the function $\tilde{f}$ obtained from $f$ by concatenating a new column with number of stones corresponding to $X_j$, then the new column will be labeled by $X_j.$ Particularly, the labeling $X_i$ is not affected by the labeling $X_j$, while the labeling $X_j$ is affected by $X_i.$ Please refer to Figure~\ref{fig_XiXj}, so we realize that $M(i,j)$ is True if $X_i=(\alpha_{11},\alpha_{11},\alpha_{02},\alpha_{02})^T$, $(\alpha_{11},\alpha_{2},\alpha_{02},\alpha_{01})^T$, or $(\alpha_{2},\alpha_{11},\alpha_{02},\alpha_{01})^T.$

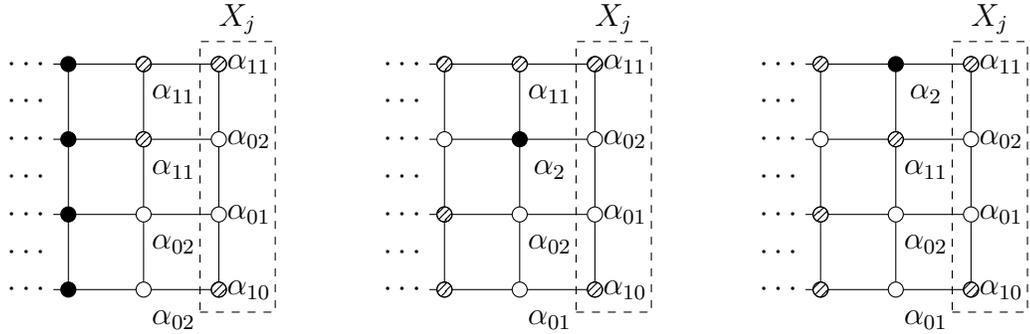
\begin{figure}[h]
\begin{center}
\begin{tikzpicture}

\draw (0.8,0)--(3,0);
\draw (0.8,1)--(3,1);
\draw (0.8,2)--(3,2);
\draw (0.8,3)--(3,3);
\draw (1,0)--(1,3);
\draw (2,0)--(2,3);
\draw (3,0)--(3,3);
\draw (0.47,0) node{$\cdots$};
\draw (0.47,0.5) node{$\cdots$};
\draw (0.47,1) node{$\cdots$};
\draw (0.47,1.5) node{$\cdots$};
\draw (0.47,2) node{$\cdots$};
\draw (0.47,2.5) node{$\cdots$};
\draw (0.47,3) node{$\cdots$};

\draw (3.4,3) node {$\alpha_{11}$};
\draw (3.4,2) node {$\alpha_{02}$};
\draw (3.4,1) node {$\alpha_{01}$};
\draw (3.4,0) node {$\alpha_{10}$};
\draw (2.4,2.6) node {$\alpha_{11}$};
\draw (2.4,1.6) node {$\alpha_{2}$};
\draw (2.4,0.6) node {$\alpha_{02}$};
\draw (2.4,-0.4) node {$\alpha_{01}$};

\draw[fill=white] (1,0)coordinate(A10) circle(0.1cm);
\draw[pattern=north east lines] (1,0)coordinate(A10) circle(0.1cm);
\draw[fill=white] (1,1)coordinate(A11) circle(0.1cm);
\draw[pattern=north east lines] (1,1)coordinate(A11) circle(0.1cm);
\draw[fill=white] (1,2)coordinate(A12) circle(0.1cm);
\draw[fill=white] (1,3)coordinate(A13) circle(0.1cm);
\draw[pattern=north east lines] (1,3)coordinate(A13) circle(0.1cm);
\draw[fill=white] (2,0)coordinate(A20) circle(0.1cm);
\draw[fill=white] (2,1)coordinate(A21) circle(0.1cm);
\draw[fill=black] (2,2)coordinate(A22) circle(0.1cm);
\draw[fill=white] (2,3)coordinate(A23) circle(0.1cm);
\draw[pattern=north east lines] (2,3)coordinate(A23) circle(0.1cm);
\draw[fill=white] (3,0)coordinate(A30) circle(0.1cm);
\draw[pattern=north east lines] (3,0)coordinate(A30) circle(0.1cm);
\draw[fill=white] (3,1)coordinate(A31) circle(0.1cm);
\draw[fill=white] (3,2)coordinate(A32) circle(0.1cm);
\draw[fill=white] (3,3)coordinate(A33) circle(0.1cm);
\draw[pattern=north east lines] (3,3)coordinate(A33) circle(0.1cm);


\draw (5.8,0)--(8,0);
\draw (5.8,1)--(8,1);
\draw (5.8,2)--(8,2);
\draw (5.8,3)--(8,3);
\draw (6,0)--(6,3);
\draw (7,0)--(7,3);
\draw (8,0)--(8,3);
\draw (5.47,0) node{$\cdots$};
\draw (5.47,0.5) node{$\cdots$};
\draw (5.47,1) node{$\cdots$};
\draw (5.47,1.5) node{$\cdots$};
\draw (5.47,2) node{$\cdots$};
\draw (5.47,2.5) node{$\cdots$};
\draw (5.47,3) node{$\cdots$};

\draw (8.4,3) node {$\alpha_{11}$};
\draw (8.4,2) node {$\alpha_{02}$};
\draw (8.4,1) node {$\alpha_{01}$};
\draw (8.4,0) node {$\alpha_{10}$};
\draw (7.4,2.6) node {$\alpha_{2}$};
\draw (7.4,1.6) node {$\alpha_{11}$};
\draw (7.4,0.6) node {$\alpha_{02}$};
\draw (7.4,-0.4) node {$\alpha_{01}$};

\draw[fill=white] (6,0)coordinate(A60) circle(0.1cm);
\draw[pattern=north east lines] (6,0)coordinate(A60) circle(0.1cm);
\draw[fill=white] (6,1)coordinate(A61) circle(0.1cm);
\draw[pattern=north east lines] (6,1)coordinate(A61) circle(0.1cm);
\draw[fill=white] (6,2)coordinate(A62) circle(0.1cm);
\draw[fill=white] (6,3)coordinate(A63) circle(0.1cm);
\draw[pattern=north east lines] (6,3)coordinate(A63) circle(0.1cm);
\draw[fill=white] (7,0)coordinate(A70) circle(0.1cm);
\draw[fill=white] (7,1)coordinate(A71) circle(0.1cm);
\draw[fill=white] (7,2)coordinate(A72) circle(0.1cm);
\draw[pattern=north east lines] (7,2)coordinate(A72) circle(0.1cm);
\draw[fill=black] (7,3)coordinate(A73) circle(0.1cm);
\draw[fill=white] (8,0)coordinate(A80) circle(0.1cm);
\draw[pattern=north east lines] (8,0)coordinate(A80) circle(0.1cm);
\draw[fill=white] (8,1)coordinate(A81) circle(0.1cm);
\draw[fill=white] (8,2)coordinate(A82) circle(0.1cm);
\draw[fill=white] (8,3)coordinate(A83) circle(0.1cm);
\draw[pattern=north east lines] (8,3)coordinate(A83) circle(0.1cm);


\draw (-4.2,0)--(-2,0);
\draw (-4.2,1)--(-2,1);
\draw (-4.2,2)--(-2,2);
\draw (-4.2,3)--(-2,3);
\draw (-4,0)--(-4,3);
\draw (-3,0)--(-3,3);
\draw (-2,0)--(-2,3);
\draw (-4.53,0) node{$\cdots$};
\draw (-4.53,0.5) node{$\cdots$};
\draw (-4.53,1) node{$\cdots$};
\draw (-4.53,1.5) node{$\cdots$};
\draw (-4.53,2) node{$\cdots$};
\draw (-4.53,2.5) node{$\cdots$};
\draw (-4.53,3) node{$\cdots$};

\draw (-1.6,3) node {$\alpha_{11}$};
\draw (-1.6,2) node {$\alpha_{02}$};
\draw (-1.6,1) node {$\alpha_{01}$};
\draw (-1.6,0) node {$\alpha_{10}$};
\draw (-2.6,2.6) node {$\alpha_{11}$};
\draw (-2.6,1.6) node {$\alpha_{11}$};
\draw (-2.6,0.6) node {$\alpha_{02}$};
\draw (-2.6,-0.4) node {$\alpha_{02}$};

\draw[fill=black] (-4,0)coordinate(A-40) circle(0.1cm);
\draw[fill=black] (-4,1)coordinate(A-41) circle(0.1cm);
\draw[fill=black] (-4,2)coordinate(A-42) circle(0.1cm);
\draw[fill=black] (-4,3)coordinate(A-43) circle(0.1cm);
\draw[fill=white] (-3,0)coordinate(A-30) circle(0.1cm);
\draw[fill=white] (-3,1)coordinate(A-31) circle(0.1cm);
\draw[fill=white] (-3,2)coordinate(A-32) circle(0.1cm);
\draw[pattern=north east lines] (-3,2)coordinate(A-32) circle(0.1cm);
\draw[fill=white] (-3,3)coordinate(A-33) circle(0.1cm);
\draw[pattern=north east lines] (-3,3)coordinate(A-33) circle(0.1cm);
\draw[fill=white] (-2,0)coordinate(A-20) circle(0.1cm);
\draw[pattern=north east lines] (-2,0)coordinate(A-20) circle(0.1cm);
\draw[fill=white] (-2,1)coordinate(A-21) circle(0.1cm);
\draw[fill=white] (-2,2)coordinate(A-22) circle(0.1cm);
\draw[fill=white] (-2,3)coordinate(A-23) circle(0.1cm);
\draw[pattern=north east lines] (-2,3)coordinate(A-23) circle(0.1cm);


\draw[dashed] (2.75,-0.3) rectangle (3.75,3.3);
\draw (3.25,3.6) node {$X_j$};
\draw[dashed] (7.75,-0.3) rectangle (8.75,3.3);
\draw (8.25,3.6) node {$X_j$};
\draw[dashed] (-2.25,-0.3) rectangle (-1.25,3.3);
\draw (-1.75,3.6) node {$X_j$};

\end{tikzpicture}

\caption{Some cases of $f\in{\cal{F}}_{n-1}(X_i)$ and $\tilde{f}\in{\cal{F}}_{n}(X_j)$, where $X_j$ is given.}
\label{fig_XiXj}

\end{center}
\end{figure}

However, in the same example above, we again emphasize that the labeling $X_j$ is affected by the labeling $X_i$. In Figure~\ref{fig_notXiXj} we see that if $X_i=(\alpha_{02},\alpha_{2},\alpha_{02},\alpha_{11})^T$ and $X_j=(\alpha_{11},\alpha_{02},\alpha_{01},\alpha_{10})^T$, then the top vertex and the bottom vertex of $X_j$ must be labeled by $\alpha_{10}$ and $\alpha_{11}$ instead of $\alpha_{11}$ and $\alpha_{10}$, respectively. Therefore, $M(i,j)$ is False.

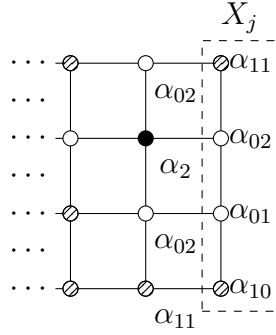
\begin{figure}[h]
\begin{center}
\begin{tikzpicture}

\draw (0.8,0)--(3,0);
\draw (0.8,1)--(3,1);
\draw (0.8,2)--(3,2);
\draw (0.8,3)--(3,3);
\draw (1,0)--(1,3);
\draw (2,0)--(2,3);
\draw (3,0)--(3,3);
\draw (0.47,0) node{$\cdots$};
\draw (0.47,0.5) node{$\cdots$};
\draw (0.47,1) node{$\cdots$};
\draw (0.47,1.5) node{$\cdots$};
\draw (0.47,2) node{$\cdots$};
\draw (0.47,2.5) node{$\cdots$};
\draw (0.47,3) node{$\cdots$};

\draw (3.4,3) node {$\alpha_{11}$};
\draw (3.4,2) node {$\alpha_{02}$};
\draw (3.4,1) node {$\alpha_{01}$};
\draw (3.4,0) node {$\alpha_{10}$};
\draw (2.4,2.6) node {$\alpha_{02}$};
\draw (2.4,1.6) node {$\alpha_{2}$};
\draw (2.4,0.6) node {$\alpha_{02}$};
\draw (2.4,-0.4) node {$\alpha_{11}$};

\draw[fill=white] (1,0)coordinate(A10) circle(0.1cm);
\draw[pattern=north east lines] (1,0)coordinate(A10) circle(0.1cm);
\draw[fill=white] (1,1)coordinate(A11) circle(0.1cm);
\draw[pattern=north east lines] (1,1)coordinate(A11) circle(0.1cm);
\draw[fill=white] (1,2)coordinate(A12) circle(0.1cm);
\draw[fill=white] (1,3)coordinate(A13) circle(0.1cm);
\draw[pattern=north east lines] (1,3)coordinate(A13) circle(0.1cm);
\draw[fill=white] (2,0)coordinate(A20) circle(0.1cm);
\draw[pattern=north east lines] (2,0)coordinate(A20) circle(0.1cm);
\draw[fill=white] (2,1)coordinate(A21) circle(0.1cm);
\draw[fill=black] (2,2)coordinate(A22) circle(0.1cm);
\draw[fill=white] (2,3)coordinate(A23) circle(0.1cm);
\draw[fill=white] (3,0)coordinate(A30) circle(0.1cm);
\draw[pattern=north east lines] (3,0)coordinate(A30) circle(0.1cm);
\draw[fill=white] (3,1)coordinate(A31) circle(0.1cm);
\draw[fill=white] (3,2)coordinate(A32) circle(0.1cm);
\draw[fill=white] (3,3)coordinate(A33) circle(0.1cm);
\draw[pattern=north east lines] (3,3)coordinate(A33) circle(0.1cm);

\draw[dashed] (2.75,-0.3) rectangle (3.75,3.3);
\draw (3.25,3.6) node {$X_j$};

\end{tikzpicture}

\caption{Not a case of $f\in{\cal{F}}_{n-1}(X_i)$ and $\tilde{f}\in{\cal{F}}_{n}(X_j)$, where $X_j$ is given.}
\label{fig_notXiXj}

\end{center}
\end{figure}

\medskip

As an application of the proposed algorithm, we obtain the values of $\gamma_{\{2\}}(G_{4,n})$ for $n\leq 100$ listed as follows.
\begin{thm}     \label{thm_Gridmn_4}
For positive integer $n\leq 100$, we have
$$\gamma_{\{2\}}(G_{4,n})=\begin{cases}
    2n+2, \quad \text{if}~n=1;\\
    2n+1, \quad \text{if}~n\in\{2,3\}\cup\ \{5,6\};\\
    2n, \quad \text{if}~n\in\{4\}\cup\{7,\dots,14\}\cup\{16,17\};\\
    2n-1, \quad \text{if}~n\in\{15\}\cup\{18,\dots,25\}\cup\{27,28\};\\
    2n-2, \quad \text{if}~n\in\{26\}\cup\{29,\dots,36\}\cup\{38\};\\
    2n-3, \quad \text{if}~n\in\{37\}\cup\{39,\dots,47\};\\
    2n-4, \quad \text{if}~n\in\{48,\dots,58\};\\
    2n-5, \quad \text{if}~n\in\{59,\dots,68\};\\
    2n-6, \quad \text{if}~n\in\{69,\dots,78\};\\
    2n-7, \quad \text{if}~n\in\{79,\dots,88\};\\
    2n-8, \quad \text{if}~n\in\{89,\dots,98\};\\
    2n-9, \quad \text{if}~n\in\{99,100\}.\\
\end{cases}$$
\end{thm}
\qed

For the domination number $\gamma(G_{4,n})=\gamma_{\{1\}}(G_{4,n}),$ existing result tells us
$$\gamma(G_{4,n})=\begin{cases}
n+1, \quad \text{if}~n\in\{1,2,3,5,6,9\};\\
n, \quad \text{otherwise}.
\end{cases}$$
Clearly, $\gamma_{\{2\}}(G_{4,n})\leq 2\gamma(G_{4,n})$. Hence, it is of interest to see $G_{4,15}$, the first case satisfying $\gamma_{\{2\}}(G_{4,n})<2n$. Also, we can track to an integer $\{2\}$- dominating function on $G_{4,15}$ of weight $29$ by the same programming as shown in Figure~\ref{fig_G415}. Recall that each vertex $v\in V_{m,n}$ is denoted by a circle filled with white, slashes, and black if the mapped value $f(v)=0,$ $1$, and $2,$ respectively.

\begin{figure}[h]
\begin{center}
\begin{tikzpicture}

\draw (0,0)--(14,0);
\draw (0,1)--(14,1);
\draw (0,2)--(14,2);
\draw (0,3)--(14,3);
\draw (0,0)--(0,3);
\draw (1,0)--(1,3);
\draw (2,0)--(2,3);
\draw (3,0)--(3,3);
\draw (4,0)--(4,3);
\draw (5,0)--(5,3);
\draw (6,0)--(6,3);
\draw (7,0)--(7,3);
\draw (8,0)--(8,3);
\draw (9,0)--(9,3);
\draw (10,0)--(10,3);
\draw (11,0)--(11,3);
\draw (12,0)--(12,3);
\draw (13,0)--(13,3);
\draw (14,0)--(14,3);

\draw[fill=white] (0,0)coordinate(A00) circle(0.1cm);
\draw[fill=white] (0,1)coordinate(A01) circle(0.1cm);
\draw[pattern=north east lines] (0,1)coordinate(A01) circle(0.1cm);
\draw[fill=white] (0,2)coordinate(A02) circle(0.1cm);
\draw[pattern=north east lines] (0,2)coordinate(A02) circle(0.1cm);
\draw[fill=white] (0,3)coordinate(A03) circle(0.1cm);
\draw[fill=white] (1,0)coordinate(A10) circle(0.1cm);
\draw[pattern=north east lines] (1,0)coordinate(A10) circle(0.1cm);
\draw[fill=white] (1,1)coordinate(A11) circle(0.1cm);
\draw[fill=white] (1,2)coordinate(A12) circle(0.1cm);
\draw[fill=white] (1,3)coordinate(A13) circle(0.1cm);
\draw[pattern=north east lines] (1,3)coordinate(A13) circle(0.1cm);
\draw[fill=white] (2,0)coordinate(A20) circle(0.1cm);
\draw[pattern=north east lines] (2,0)coordinate(A20) circle(0.1cm);
\draw[fill=white] (2,1)coordinate(A21) circle(0.1cm);
\draw[fill=white] (2,2)coordinate(A22) circle(0.1cm);
\draw[fill=white] (2,3)coordinate(A23) circle(0.1cm);
\draw[pattern=north east lines] (2,3)coordinate(A23) circle(0.1cm);
\draw[fill=white] (3,0)coordinate(A30) circle(0.1cm);
\draw[fill=white] (3,1)coordinate(A31) circle(0.1cm);
\draw[pattern=north east lines] (3,1)coordinate(A31) circle(0.1cm);
\draw[fill=white] (3,2)coordinate(A32) circle(0.1cm);
\draw[pattern=north east lines] (3,2)coordinate(A32) circle(0.1cm);
\draw[fill=white] (3,3)coordinate(A33) circle(0.1cm);
\draw[fill=white] (4,0)coordinate(A40) circle(0.1cm);
\draw[fill=white] (4,1)coordinate(A41) circle(0.1cm);
\draw[fill=white] (4,2)coordinate(A42) circle(0.1cm);
\draw[pattern=north east lines] (4,2)coordinate(A42) circle(0.1cm);
\draw[fill=white] (4,3)coordinate(A43) circle(0.1cm);
\draw[fill] (5,0)coordinate(A50) circle(0.1cm);
\draw[fill=white] (5,1)coordinate(A51) circle(0.1cm);
\draw[fill=white] (5,2)coordinate(A52) circle(0.1cm);
\draw[fill=white] (5,3)coordinate(A53) circle(0.1cm);
\draw[pattern=north east lines] (5,3)coordinate(A53) circle(0.1cm);
\draw[fill=white] (6,0)coordinate(A60) circle(0.1cm);
\draw[fill=white] (6,1)coordinate(A61) circle(0.1cm);
\draw[fill=white] (6,2)coordinate(A62) circle(0.1cm);
\draw[fill=white] (6,3)coordinate(A63) circle(0.1cm);
\draw[pattern=north east lines] (6,3)coordinate(A63) circle(0.1cm);
\draw[fill=white] (7,0)coordinate(A70) circle(0.1cm);
\draw[fill] (7,1)coordinate(A71) circle(0.1cm);
\draw[fill=white] (7,2)coordinate(A72) circle(0.1cm);
\draw[pattern=north east lines] (7,2)coordinate(A72) circle(0.1cm);
\draw[fill=white] (7,3)coordinate(A73) circle(0.1cm);
\draw[fill=white] (8,0)coordinate(A80) circle(0.1cm);
\draw[fill=white] (8,1)coordinate(A81) circle(0.1cm);
\draw[fill=white] (8,2)coordinate(A82) circle(0.1cm);
\draw[fill=white] (8,3)coordinate(A83) circle(0.1cm);
\draw[pattern=north east lines] (8,3)coordinate(A83) circle(0.1cm);
\draw[fill] (9,0)coordinate(A90) circle(0.1cm);
\draw[fill=white] (9,1)coordinate(A91) circle(0.1cm);
\draw[fill=white] (9,2)coordinate(A92) circle(0.1cm);
\draw[fill=white] (9,3)coordinate(A93) circle(0.1cm);
\draw[pattern=north east lines] (9,3)coordinate(A93) circle(0.1cm);
\draw[fill=white] (10,0)coordinate(A100) circle(0.1cm);
\draw[fill=white] (10,1)coordinate(A101) circle(0.1cm);
\draw[fill=white] (10,2)coordinate(A102) circle(0.1cm);
\draw[pattern=north east lines] (10,2)coordinate(A102) circle(0.1cm);
\draw[fill=white] (10,3)coordinate(A103) circle(0.1cm);
\draw[fill=white] (11,0)coordinate(A110) circle(0.1cm);
\draw[fill=white] (11,1)coordinate(A111) circle(0.1cm);
\draw[pattern=north east lines] (11,1)coordinate(A111) circle(0.1cm);
\draw[fill=white] (11,2)coordinate(A112) circle(0.1cm);
\draw[pattern=north east lines] (11,2)coordinate(A112) circle(0.1cm);
\draw[fill=white] (11,3)coordinate(A113) circle(0.1cm);
\draw[fill=white] (12,0)coordinate(A120) circle(0.1cm);
\draw[pattern=north east lines] (12,0)coordinate(A120) circle(0.1cm);
\draw[fill=white] (12,1)coordinate(A121) circle(0.1cm);
\draw[fill=white] (12,2)coordinate(A122) circle(0.1cm);
\draw[fill=white] (12,3)coordinate(A123) circle(0.1cm);
\draw[pattern=north east lines] (12,3)coordinate(A123) circle(0.1cm);
\draw[fill=white] (13,0)coordinate(A130) circle(0.1cm);
\draw[pattern=north east lines] (13,0)coordinate(A130) circle(0.1cm);
\draw[fill=white] (13,1)coordinate(A131) circle(0.1cm);
\draw[fill=white] (13,2)coordinate(A132) circle(0.1cm);
\draw[fill=white] (13,3)coordinate(A133) circle(0.1cm);
\draw[pattern=north east lines] (13,3)coordinate(A133) circle(0.1cm);
\draw[fill=white] (14,0)coordinate(A140) circle(0.1cm);
\draw[fill=white] (14,1)coordinate(A141) circle(0.1cm);
\draw[pattern=north east lines] (14,1)coordinate(A141) circle(0.1cm);
\draw[fill=white] (14,2)coordinate(A142) circle(0.1cm);
\draw[pattern=north east lines] (14,2)coordinate(A142) circle(0.1cm);
\draw[fill=white] (14,3)coordinate(A143) circle(0.1cm);

\end{tikzpicture}

\caption{A dominating function on $G_{4,15}$ of weight 29.} 
\label{fig_G415}

\end{center}
\end{figure}

The value of $\gamma_{\{2\}}(G_{4,n})$ seems to be regular as $n$ increases. Hence, we observe the result in Theorem~\ref{thm_Gridmn_4} and give a conjecture.
\begin{conj}  \label{conj_gamma2G4n}
If $n\geq 49$, then
$$\gamma_{\{2\}}(G_{4,{n}})=2n-\left\lfloor\frac{n-9}{10}\right\rfloor.$$
\end{conj}

\section{Concluding remark}

For positive integers $m$ and $n$, the grid graph $G_{m,n}=P_m\square P_n$ is defined as the Cartesian product of path graphs $P_m$ of $m$ vertices and $P_n$ of $n$ vertices. An integer $\{2\}$-dominating function of $G_{m,n}$ is a function $f:V_{m,n}\rightarrow\{0,1,2\}$ satisfying $\sum_{u\in N[v]} f(u)\geq 2$ for all $v\in V_{m,n}$, where $V_{m,n}$ is the vertex set of $G_{m,n}$ and $N[v]$ is the set including $v$ and vertices adjacent to $v.$ The integer $\{2\}$-domination number $\gamma_{\{2\}}(G_{m,n})$ is the minimum value of $\sum_{v\in V_{m,n}}f(v)$ among all integer $\{2\}$-dominating functions of $G_{m,n}.$ We focus on calculating the value of $\gamma_{\{2\}}(G_{m,n})$ in this paper.

In Section~2, we propose results for $m\in\{1,2,3\}.$ In Theorem~\ref{thm_Gridmn_1}, we show
$$\gamma_{\{2\}}(G_{1,n})=2\cdot\left\lceil\frac{n}{3}\right\rceil.$$
In Theorem~\ref{thm_Gridmn_2}, we show
$$\gamma_{\{2\}}(G_{2,n})=n+1.$$
In Theorem~\ref{thm_Gridmn_3}, we show
$$\gamma_{\{2\}}(G_{3,n})\leq \left\lfloor\frac{3n}{2}\right\rfloor+1.$$
Furthermore, we suggest in Conjecture~\ref{conj_gamma2G3n} that the above inequality is actually an equality and give some comment.

In Section~3, we propose an algorithm to compute $\gamma_{\{2\}}(G_{m,n})$ for arbitrary positive integers $m$ and $n$. Moreover, if $m$ is fixed, then its time complexity will be linear in $n.$ In order to implement the algorithm, we define a labeling method which classifies the status of each vertex corresponding to any given function $f:V_{m,n}\rightarrow\{0,1,2\}.$ Additionally, we apply the algorithm on counting $\gamma_{\{2\}}(G_{4,n})$, and find that $\gamma_{\{2\}}(G_{4,n})$ has some regulation as $n\geq 49,$ that could not be easily achieved by hand calculation. We give Conjecture~\ref{conj_gamma2G4n} that
$$\gamma_{\{2\}}(G_{4,{n}})=2n-\left\lfloor\frac{n-9}{10}\right\rfloor,\quad\quad\quad\text{for}~n\geq 49,$$

As a future work, we believe that the algorithm would be generalized to find $\gamma_{\{k\}}(G_{m,n})$ for arbitrary positive integer $k$. The methods given in~\cite{gprt11} and~\cite{rt19} were probably key in this topic.

\section*{Acknowledgments}
This research is supported by the National Science and Technology Council of Taiwan R.O.C. under the project NSTC 113-2115-M-031-005.


\end{document}